\documentclass[10pt]{article}

\usepackage{amsmath,amsthm}
\usepackage{amssymb}
\usepackage{color}
\usepackage{breqn}
\usepackage{enumerate}
\usepackage{graphicx}
\usepackage{bbm}
\usepackage[affil-it]{authblk}

\newtheorem{theorem}{Theorem}
\newtheorem{corollary}[theorem]{Corollary}
\newtheorem{lemma}[theorem]{Lemma}
\newtheorem{proposition}[theorem]{Proposition}

\newtheorem{conjecture}[theorem]{Conjecture}

\theoremstyle{definition}

\newcommand{\tba}{triangle blocking arrangement}
\newcommand{\tbas}{triangle blocking arrangements}

\newcommand{\trt}{\mathbf{T}}
\newcommand{\tbty}[2]{\mathbf{#1_{#2}}}
\newcommand{\intr}{\operatorname{int}}
\newcommand{\ct}{classification theorem}

\makeatletter
\def\blfootnote{\gdef\@thefnmark{}\@footnotetext}
\makeatother	
	
\begin{document}
\title{Classification theorem for strong triangle blocking arrangements}
\date{}
\author{Luka Mili\'cevi\'c
\thanks{E-mail address: \texttt{luka.milicevic@turing.mi.sanu.ac.rs}}}
\affil{Mathematical Institute of the Serbian Academy of Sciences and Arts}
\maketitle

\abstract{A strong triangle blocking arrangement is a geometric arrangement of some line segments in a triangle with certain intersection properties. It turns out that they are closely related to blocking sets. Our aim in this paper is to prove a classification theorem for strong triangle blocking arrangements. As an application, we obtain a new proof of the result of Ackerman, Buchin, Knauer, Pinchasi and Rote which says that $n$ points in general position cannot be blocked by $n-1$ points, unless $n = 2,4$. We also conjecture an extremal variant of the blocking points problem.}

\section{Introduction}\label{sectionIntro}

\hspace{12pt}We start with some background. In~\cite{ErdosPurdy}, Erd\H{o}s and Purdy posed the following problem. Given a set $P$ of $n$ points in the plane, not all on a line, how many new points (different from points in $P$) have to be chosen so that every line spanned by $P$ meets a new point? They conjectured that the answer is at least $(1 + o(1))n$ lines. We refer to the set $P$ as the \emph{initial points} and to the newly chosen points as the \emph{blocking points}. We also say that a blocking point $x$ \emph{blocks} a line $l$ if $x \in l$. The current record is due to Pinchasi~\cite{Pinch}.  

\begin{theorem}[Pinchasi~\cite{Pinch}] Given a set $P$ of $n$ points in the plane, not all on a line, we need at least $\frac{n-1}{3}$ blocking points to block all lines spanned by $P$.\end{theorem}

It is worth mentioning that the first bounds of the form $\Omega(n)$ were due to Szemer\'edi and Trotter~\cite{SzTrotter} and Beck~\cite{Beck}. Those papers are in fact about Dirac's conjecture, which states that for a set $\mathcal{L}$ of $n$ non-concurrent lines in the plane, there is a line $l \in \mathcal{L}$ that forms at least $\frac{n}{2} - O(1)$ different intersection points with lines in $\mathcal{L}$. Szemer\'edi and Trotter, and Beck, prove that there is a line with $\Omega(n)$ different intersection points, which implies in particular that at least $\Omega(n)$ blocking points are necessary. Let us also remark that in~\cite{SzTrotter}, Szemer\'edi and Trotter prove their celebrated theorem on the number of incidence between lines and points, and they use this theorem to deduce the weak form of Dirac's conjecture. On the other hand, Pinchasi's proof is purely combinatorial and avoids any use of the incidence theorem.\\

In the same paper~\cite{ErdosPurdy}, Erd\H{o}s and Purdy propose a variant of the question where $P$ has to be in general position (i.e. no three points are collinear), but remark that Gr\"{u}nbaum pointed out to them that $2\lfloor\frac{n}{2}\rfloor$ new points suffice. It is easy to see that we need at least $n$ new points when $n$ is odd, and $n-1$ new points when $n$ is even, so this remark actually has a typo, and should probably be $2\lceil\frac{n}{2} \rceil -1$. However, with this obvious typo corrected, what Erd\H{o}s and Purdy actually imply is that there are examples where $n-1$ new points are sufficient, for infinitely many $n$. Therefore, the right question is: when do $n-1$ new points suffice? This was answered by Ackerman, Buchin, Knauer, Pinchasi and Rote in~\cite{magic}.

\begin{theorem}[Ackerman, Buchin, Knauer, Pinchasi and Rote~\cite{magic}]\label{blockingTheorem} Let $P$ be a set of $n \geq 5$ points in general position and let $B$ be a set of some other points in the plane, such that every line determined by two points in $P$ meets a point in $B$. Then $|B| \geq n$.\end{theorem}

Observe immediately that the theorem is trivial when $n$ is odd. Indeed, if $n$ is odd, each point in $B$ can meet at most $\lfloor \frac{n}{2} \rfloor = \frac{n-1}{2}$ lines spanned by $P$, so at least $n$ points are needed. (If $n$ is even, the same counting argument gives only the bound of $n-1$.) On the other hand, the fact that there is a non-trivial example for $n=4$ shows that we cannot hope for such a short argument in the general case. Similarly, the theorem may naturally be compared to the Sylvester-Gallai theorem (posed by Sylvester~\cite{SylvesterProblem} and solved by Gallai~\cite{GallaiSolution}), but once again, the $n=4$ case tells us that we can expect the proof to be more involved than, for example, looking at the minimal height of a triangular region formed in the dual (as in the usual proof of the Sylvester-Gallai theorem).\\
\indent Let us also remark that the regular $n$-gon in the projective plane with $n$ points on the line at infinity corresponding to directions of diagonals, show that for every $n$, $n$ blocking points suffice for certain configurations. Of course, using suitable transformations, we can make such examples affine.\\ 

Our proof of Theorem~\ref{blockingTheorem} is based on the following classification theorem (Theorem~\ref{informalClassification}), which is the main result of this paper. At this stage, we state this theorem informally, as otherwise the definitions would occupy significant part of the introduction. For an arbitrary convex polygon $R$, we write $\partial R$ for the \emph{boundary} of $R$, which is the union of its edges.\\
\indent Before we state the classification theorem, let us very briefly explain what the theorem is about. Namely, we consider a triangular region which has two sets of segments $\mathcal{S}$ and $\mathcal{B}$ inside it, with vertices on the boundary of the triangle. These have the property that whenever two segments in $\mathcal{S}$ intersect, then there is a unique segment in $\mathcal{B}$ passing through their intersection. Also, when a segment in $\mathcal{S}$ meets a segment in $\mathcal{B}$, then there is a unique other segment in $\mathcal{S}$ passing through their intersection. These two conditions come from considering the dual of the hypothetical extremal arrangement of points in general position and their blocking points, and are requirements \textbf{(i)} and \textbf{(ii)} in the classification theorem. Our aim is
to classify all such collections of segments that satisfy an additional condition. This is the condition \textbf{(iii)} in the statement. Let us remark here that although this condition looks somewhat artificial compared to the other two, it actually develops naturally in the proof of Theorem~\ref{blockingTheorem}.

\begin{theorem}\label{informalClassification} Let $T$ be a triangle, with edges $e_1, e_2, e_3$, and let $\mathcal{S}$ and $\mathcal{B}$ be collections of segments inside $T$, with endpoints on the edges of $T$, but no internal point of a segment meets $\partial T$. Write $\overline{\mathcal{S}} = \mathcal{S} \cup \{e_1, e_2, e_3\}$. Suppose that
\begin{itemize}
\item[\textbf{(i)}] No three segments of $\overline{\mathcal{S}}$ are concurrent, and for any two such segments that intersect at a point $p$, there is a unique segment $\beta(p) \in \mathcal{B}$ that passes through $p$, except possibly when the two segments are edges of $T$, in which case there might not be any such segment in $\mathcal{B}$.
\item[\textbf{(ii)}] For every intersection $p$ of a segment in $\overline{\mathcal{S}}$ and a segment in $\mathcal{B}$, there is a unique second segment in $\overline{\mathcal{S}}$ that passes through $p$.
\item[\textbf{(iii)}] In every minimal $\overline{\mathcal{S}}$-region $R$, for any consecutive vertices $v_1, v_2, v_3, v_4, v_5$ appearing in this order on $\partial R$, we have that, if $l(v_1v_2)$ and $\beta(v_3)$ intersect in $T$, and $\beta(v_3)$ crosses the interior of $R$, then $l(v_1v_2), \beta(v_3), l(v_4v_5)$ are concurrent.\end{itemize}
Then the configuration formed by $T, \mathcal{S}$ and $\mathcal{B}$ must have one of the structures shown in the Figure~\ref{structuresFigure}.
\end{theorem}

\begin{figure}
\includegraphics[width=\textwidth]{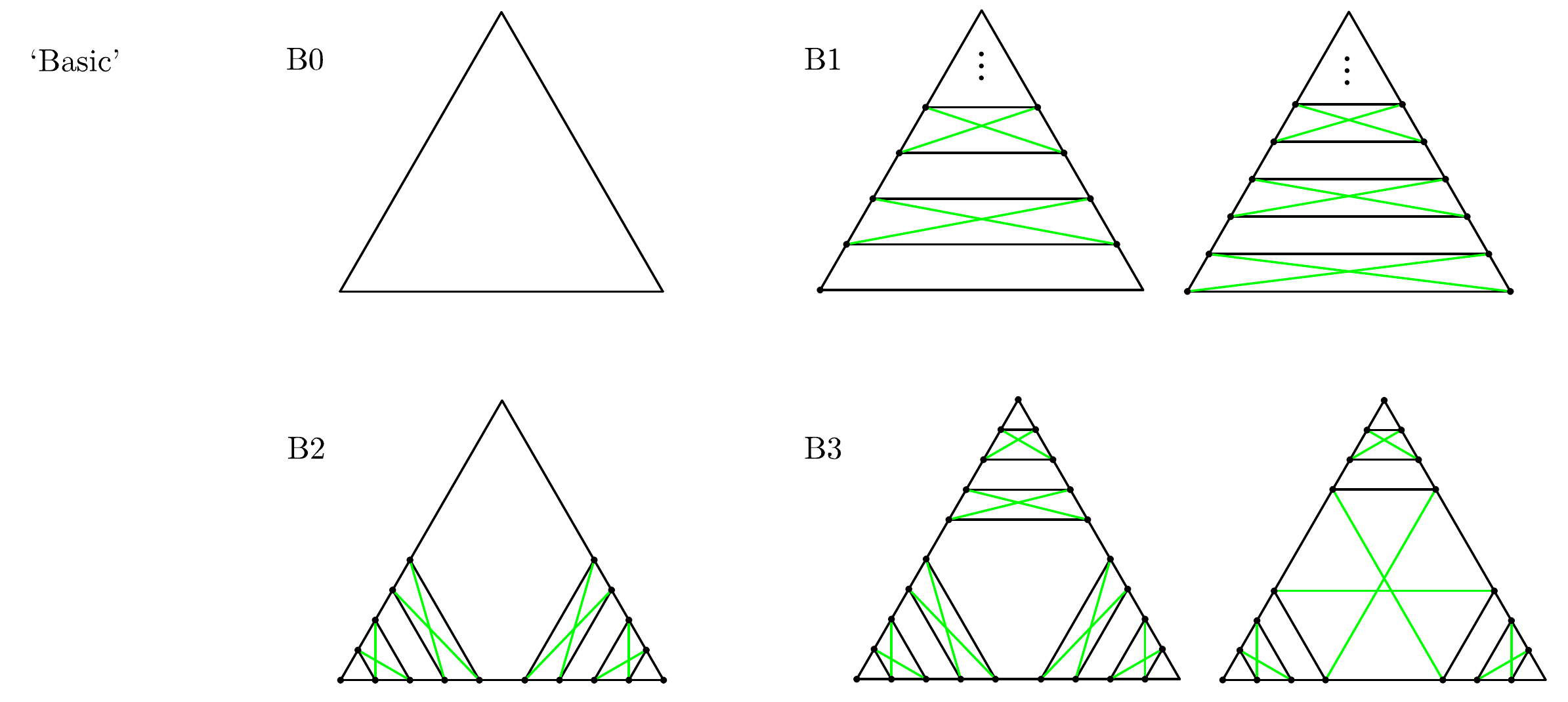}\\[24pt]
\includegraphics[width=\textwidth]{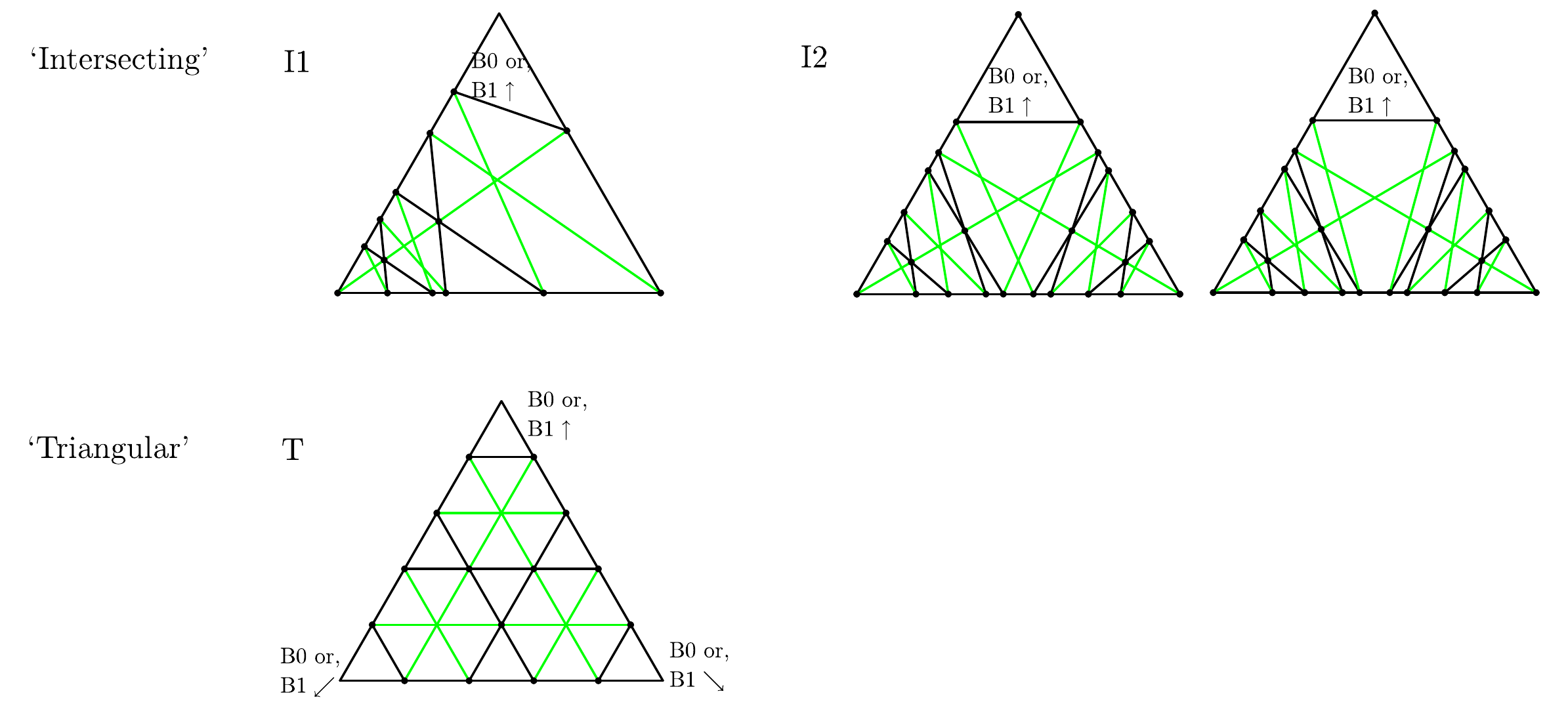}
\caption{Types of \tbas}
\label{structuresFigure}
\end{figure}

We call an arrangement satisfying conditions \textbf{(i)} and \textbf{(ii)} a \emph{triangle blocking arrangement}. If a \tba~additionally satisfies the condition \textbf{(iii)}, then we call it a \emph{strong triangle blocking arrangement}.

Given that there is such a strong structure theorem in this setting, it is plausible that an extremal result could hold. Recall that Theorem~\ref{blockingTheorem} resembles the Sylvester-Gallai theorem, which has its extremal version in the following theorem of Green and Tao. For a given set of points in the plane, we say that a line is \emph{ordinary} if it passes through exactly two points in the set.

\begin{theorem}[Green and Tao~\cite{GreenTao}] \label{GTord}There is an $n_0$ such that, whenever we have $n \geq n_0$ points in the plane that span at most $\frac{n}{2}$ ordinary lines, there is a cubic curve containing the given points.\end{theorem}

With this in mind, we formulate the following conjecture.

\begin{conjecture} Suppose that $P$ is a set of $n$ points in the plane in general position, and let $B$ be a set of blocking points for $P$. If $|B| = n$, then $P \cup B$ lie on a cubic curve.\end{conjecture}

We postpone the discussion of the connection between proof of Theorem~\ref{GTord} and the classification theorem (Theorem~\ref{informalClassification}) we prove here to the concluding remarks. There we also discuss why the condition \textbf{(iii)} is necessary in the classification theorem.

The plan of the paper is as follows. In the next section we describe the classification theorem. Then, in Section~\ref{sectionDeduct}, we see how to deduce Theorem~\ref{blockingTheorem} from the classification theorem. In Section~\ref{sectionClassificationProof}, we prove the classification theorem. This is actually the main part of the paper.

\section{Detailed description of the structural theorem}\label{sectionDesc}

\hspace{12pt}Before stating the classification result, we first need to introduce some terminology. Recall that \emph{\tba}~is a triple $\Delta = (T, \mathcal{S}, \mathcal{B})$ consisting of a triangle $T$, with vertices $x_1, x_2$ and $x_3$, a two collections of segments $\mathcal{S}$ and $\mathcal{B}$ such that the endpoints of each segment lie on the boundary of $\partial T$ (possibly coinciding with some vertex $x_i$), and no interior point of a segment lies on the boundary $\partial T$, and the following intersection condition is satisfied: for every pair of segments in $\overline{\mathcal{S}} \colon= \mathcal{S} \cup \{x_1 x_2, x_2 x_3, x_3 x_1\}$, except possibly the pairs of sides of $T$, if they intersect, there is a unique segment in $\mathcal{B}$ that passes through their intersection, and for every intersecting pair of segments, where one segment is in $\overline{\mathcal{S}}$ and the other in $\mathcal{B}$, there is a unique second segment in $\overline{\mathcal{S}}$ that passes through their intersection. We call the elements of $\overline{\mathcal{S}}$ the \emph{initial segments}, the elements of $\mathcal{S}$ the \emph{proper initial segments} when we have to distinguish them from $\mathcal{S}$, and the elements of $\mathcal{B}$ the \emph{blocking segments}. Furthermore, if $x, y$ are two points on a segment $s \in \overline{\mathcal{S}} \cup \mathcal{B}$, we also say that $xy$ is \emph{initial segment} if $s \in \overline{\mathcal{S}}$, that $xy$ is \emph{proper initial segment} if $s \in \mathcal{S}$, and that $xy$ is \emph{blocking segment} if $s \in \mathcal{B}$. We refer to intersections of initial segments as \emph{vertices} in $T$. We will write $\beta(v)$ for the unique blocking segment through the vertex $v$, and in general, for any two points $x,y$ we write $l(xy)$ for the line through $x$ and $y$.\\
\indent Given a triangle $T'$, whose vertices lie in $T$ (possibly on the edges of $T$), and whose edges are subsets of segments in $\overline{\mathcal{S}}$, we define $\overline{\mathcal{S}_{T'}} \colon= \{T' \cap s\colon s \in \overline{\mathcal{S}}\}$, $\mathcal{B}_{T'} \colon= \{T' \cap s\colon s \in \mathcal{B}\}$ and $\mathcal{S}_{T'} \colon= \overline{\mathcal{S}_{T'}} \setminus \{\text{edges of }T'\}$. We say that $\Delta' = (T', \mathcal{S}_{T'}, \mathcal{B}_{T'})$ is a \emph{sub-\tba}~of $\Delta$, \emph{induced by} $T'$. One can check that sub-\tba~is itself a \tba.\\

We now define some special types of \tbas. Let us stress that these do not include all possible \tbas. An example not included in definitions is shown in the Figure~\ref{TBAnotInDef}.\\[6pt] 
\begin{figure}
\begin{center}
\includegraphics[width = 0.7\textwidth]{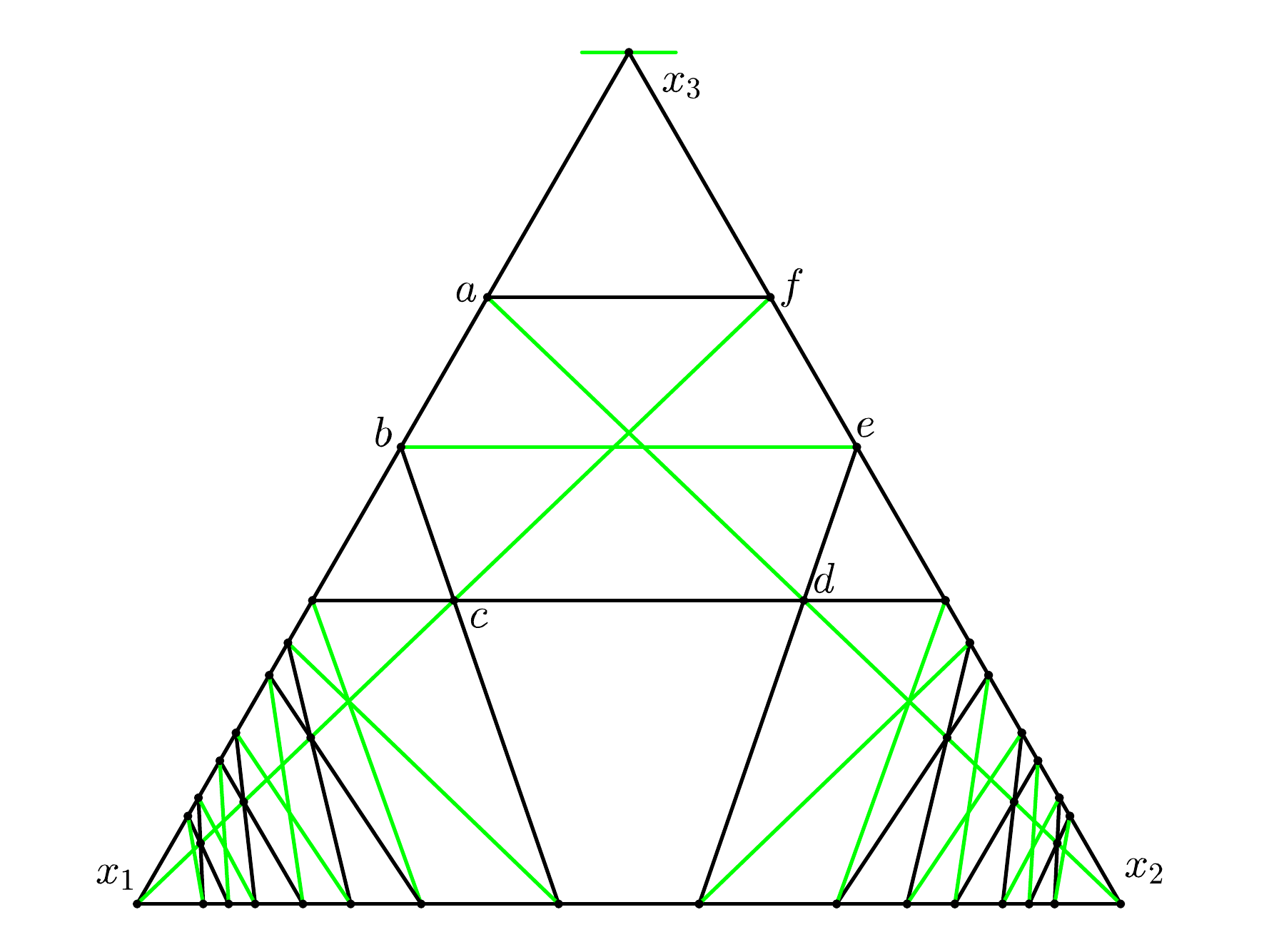}
\caption{Example of a~\tba~not among definitions}
\label{TBAnotInDef}
\end{center}
\end{figure}
\noindent\textbf{Basic types }$\tbty{B}{0}, \tbty{B}{1}, \tbty{B}{2}, \tbty{B}{3}$\textbf{.} We say a \tba~$\Delta = (T, \mathcal{S}, \mathcal{B})$ is of type $\tbty{B}{0}$ if $\mathcal{S} = \mathcal{B} = \emptyset$.\\ 

We say a \tba~$\Delta = (T, \mathcal{S}, \mathcal{B})$ is of type $\tbty{B}{1}$ if it has the following structure. There is an ordering $y_1, y_2, y_3$ of $x_1, x_2, x_3$ such that the vertices of $\mathcal{S} \cup \mathcal{B}$ are $v_1, v_2, \dots,$ $v_n \in y_1y_2$ and $u_1, u_2, \dots,$ $u_n \in y_1y_3$, with ordering $y_1, v_1, v_2, \dots,$ $v_n, y_2$ on $y_1y_2$ and ordering $y_1, u_1, u_2, \dots,$ $u_n, y_3$ on $y_1y_3$, such that $\mathcal{S} = \{v_1u_1, v_2u_2, \dots,$ $v_nu_n\}$ and $\mathcal{B} = \{v_1 u_2,$ $v_2 u_1,$ $v_3u_4,$ $v_4u_3, \dots,$ $v_{t-1} u_t,$ $v_t u_{t-1}\}$, where $t = n$ if $n$ is even, and $t = n+1$, if $n$ is odd, and in that case, we set $v_{n+1} = y_2, u_{n+1} = y_3$. We say that $y_1$ is the \emph{first vertex} of $\Delta$. Also, when $n$ is even, we say that $\Delta$ is \emph{even}, and if $n$ is odd, we say that $\Delta$ is \emph{odd}.\\

We say a \tba~$\Delta = (T, \mathcal{S}, \mathcal{B})$ is of type $\tbty{B}{2}$ if it has the following structure. There is an ordering $y_1, y_2, y_3$ of $x_1, x_2, x_3$ such that the vertices of $\mathcal{S} \cup \mathcal{B}$ are $v_1, v_2, \dots, v_n \in y_1y_2$, then $u_1, u_2, \dots,$ $u_n,$ $u'_m, \dots,$ $u'_1 \in y_1y_3$ and $w_1, w_2, \dots,$ $w_m \in y_3y_2$, with ordering $y_1, v_1, v_2, \dots,$ $v_n, y_2$ on $y_1y_2$, ordering $y_1, u_1, u_2, \dots,$ $u_n, u'_m, \dots,$ $u'_1, y_3$ on $y_1y_3$, and ordering $y_3, w_1, \dots,$ $w_m, y_2$ on $y_3y_2$, and $m$ and $n$ are even. The segments are $\mathcal{S} = \{v_1u_1, v_2u_2, \dots,$ $v_nu_n,$ $u'_1w_1, \dots,$ $u'_mw_m\}$ and $\mathcal{B} = \{v_1 u_2, v_2 u_1,$ $v_3u_4, v_4u_3, \dots,$ $v_{n-1} u_n, v_n u_{n-1}\}$ $\cup$ $\{u'_1w_2, u'_2w_1, \dots,$ $u'_nw_{n-1}, u'_{n-1}w_n\}$.\\

We say a \tba~$\Delta = (T, \mathcal{S}, \mathcal{B})$ is of type $\tbty{B}{3}$ if it has the following structure. There is an ordering $y_1, y_2, y_3$ of $x_1, x_2, x_3$ such that the vertices of $\mathcal{S} \cup \mathcal{B}$ are $v_1, \dots, v_k,$ $v'_l, \dots, v'_1 \in y_1y_2$, $w_1, \dots, w_l,$ $w'_m, \dots,$ $w'_1 \in y_2y_3$, $u_1, \dots, u_m,$ $u'_k, \dots,$ $u'_1 \in y_3y_1$, appearing in orders $y_1, v_1, \dots, v_k,$ $v'_l, \dots, v'_1, y_2$ on $y_1y_2$, $y_2, w_1, \dots, w_l,$ $w'_m, \dots, w'_1, y_3$ on $y_2y_3$, and $y_3, u_1, \dots, u_m,$ $u'_k, \dots, u'_1, y_1$ on $y_3y_1$, and $k,l,m$ are of the same parity.\\
\indent When $k,l,m$ are even, then the segments are $\mathcal{S} = \{v_1u'_1, \dots,$ $v_ku'_k\} \cup$ $\{w_1v'_1, \dots,$ $w_lv'_l\} \cup$ $\{u_1w'_1, \dots,$ $u_mw'_m\}$ and $\mathcal{B} = \{v_1u'_2,$ $v_2u'_1, \dots,$ $v_{k-1}u'_k, v_ku'_{k-1}\}$ $\cup$ $\{w_1v'_2, w_2v'_1, \dots,$ $w_{l-1}v'_l, w_l v'_{l-1}\}$ $\cup$ $\{u_1w'_2, u_2w'_1, \dots,$ $u_{m-1}w'_m, u_m w'_{m-1}\}$.\\
\indent When $k,l,m$ are odd, then the segments are $\mathcal{S} = \{v_1u'_1, \dots,$ $v_ku'_k\}$ $\cup$ $\{w_1v'_1, \dots,$ $w_lv'_l\}$ $\cup$ $\{u_1w'_1, \dots,$ $u_mw'_m\}$ and $\mathcal{B} = \{v_1u'_2,$ $v_2u'_1, \dots,$ $v_{k-1}u'_{k-2},$ $v_{k-2}u'_{k-1}\}$ $\cup$ $\{w_1v'_2, w_2v'_1, \dots,$ $w_{l-1}v'_{l-2},$ $w_{l-2} v'_{l-1}\}$ $\cup$ $\{u_1w'_2,$ $u_2w'_1, \dots,$ $u_{m-1}w'_{m-2},$ $u_{m-2} w'_{m-1}\}$ $\cup$ $\{v_kw'_m,$ $v'_lu_m, w_l u'_k\}$.\\

\noindent\textbf{Intersecting types }$\tbty{I}{1}, \tbty{I}{2}$\textbf{.} We say a \tba~$\Delta = (T, \mathcal{S}, \mathcal{B})$ is of type $\tbty{I}{1}$ if it has the following structure. There is an ordering $y_1, y_2, y_3$ of $x_1, x_2, x_3$ and there are vertices $z_2 \in y_1y_2, z_3 \in y_1y_3$ such that
\begin{enumerate}
\item The segment $z_2z_3$ is in $\mathcal{S}$, and the sub-\tba in the triangle $T' \colon= y_1z_2z_3$ is of type $\tbty{B}{0}$ or $\tbty{B}{1}$, with $y_1$ as the first vertex and it is even.
\item There are vertices $v_1, v_2, \dots, v_{2k}$ on $y_2z_2$, appearing in that order from $y_2$ to $z_2$, and there are vertices $u_1, u_2, \dots, u_{2k}$ on $y_2y_3$, appearing in that order from $y_2$ to $y_3$, such that
\[\mathcal{S} \setminus \overline{\mathcal{S}_{T'}} = \{u_1v_2, u_2v_1, u_3v_4, u_4v_3, \dots, u_{2k - 1}v_{2k}, u_{2k}v_{2k-1}\}\]
and
\[\begin{split}\mathcal{B} \setminus \mathcal{B}_{T'} = &\{u_1 v_1, u_2v_3, u_3v_2, u_4v_5, u_5v_4, \dots, u_{2k-2}v_{2k-1}, u_{2k-1}v_{2k-2}\} \\\hspace{1cm}&\cup \{y_2z_3, y_3 v_{2k}, z_2u_{2k}\}.\end{split}\]
\item For $i=1,2, \dots, k$, the segments $u_{2i-1}v_{2i}, u_{2i}v_{2i-1}, y_2z_3$ are concurrent. Let $p_i$ be their intersection point. The intersections between $\overline{\mathcal{S}}$ and $\mathcal{B}$, and the intersections of pairs of segments in $\overline{\mathcal{S}}$ are either on $\partial T$, or the points $p_i$.
\end{enumerate} 

We say a \tba~$\Delta = (T, \mathcal{S}, \mathcal{B})$ is of type $\tbty{I}{2}$ if it has the following structure. There is an ordering $y_1, y_2, y_3$ of $x_1, x_2, x_3$ and there are vertices $z_2 \in y_1y_2, z_3 \in y_1y_3$ such that
\begin{enumerate}
\item The segment $z_2z_3$ is in $\mathcal{S}$, and the sub-\tba~in the triangle $T' \colon= y_1z_2z_3$ is of type $\tbty{B}{0}$ or $\tbty{B}{1}$, with $y_1$ as the first vertex and it is even.
\item The vertices from $y_2$ to $z_2$ are $u_1, u_2, \dots, u_{2k}$, in that order, from $y_2$ to $y_3$ are $v_1, v_2, \dots, v_{2k}, w_{2l}, w_{2l-1}, \dots, w_1$, in that order, and from $y_3$ to $z_3$ are $t_1, t_2, \dots, t_{2l}$, in that order. The segments of $\Delta$ are
\[\begin{split}\mathcal{S} \setminus \overline{\mathcal{S}_{T'}} = &\{u_1 v_2, u_2 v_1, \dots, u_{2k-1}v_{2k}, u_{2k}v_{2k-1}\} \\
&\hspace{1cm}\cup \{w_1t_2, w_2t_1, \dots, w_{2l-1}t_{2l}, w_{2l}t_{2l-1}\}\end{split}\]
and
\[\begin{split}\mathcal{B} \setminus \mathcal{B}_{T'} = &\{u_1 v_1, u_2v_3, u_3v_2,\dots, u_{2k-2}v_{2k-1}, u_{2k-1}v_{2k-2}\}\\
&\hspace{1cm}\cup\{w_1 t_1, w_2t_3, w_3t_2, \dots, w_{2k-2}t_{2k-1}, w_{2k-1}t_{2k-2}\}\\
&\hspace{1cm} \cup \{u_{2k}y_3, t_{2l}y_2\} \cup \mathcal{B}',\end{split}\]
where $\mathcal{B}' = \{v_{2k}z_3, w_{2l} z_2\}$ or $\mathcal{B}' = \{v_{2k}z_2, w_{2l} z_3\}$. 
\item For $i=1,2, \dots, k$, the segments $u_{2i-1}v_{2i}, u_{2i}v_{2i-1}, y_2t_{2l}$ are concurrent, at point $p_i$, and for $i=1,2, \dots, l$, the segments $w_{2i-1}t_{2i}, w_{2i}t_{2i-1}, y_3u_{2k}$ are concurrent, at point $q_i$. The intersections between segments in $\overline{\mathcal{S}}$ and $\mathcal{B}$, and the intersections of pairs of segments in $\overline{\mathcal{S}}$ are either on $\partial T$, or the points $p_i$ and $q_i$.
\end{enumerate}

\noindent\textbf{Triangular type }$\trt$\textbf{.} We say a \tba~$\Delta = (T, \mathcal{S}, \mathcal{B})$ is of type $\trt$ if it has the following structure. There is an integer $k \geq 2$ and there are vertices $u_1, u_2, \dots, u_{2k} \in x_1x_2$, appearing in that order from $x_1$ to $x_2$, vertices $v_1, v_2, \dots, v_{2k} \in x_2x_3$, appearing in that order from $x_2$ to $x_3$, and vertices $w_1, w_2, \dots, w_{2k} \in x_3x_1$, appearing in that order from $x_3$ to $x_1$.
\begin{enumerate}
\item Segments $u_1 w_{2k}, v_1 u_{2k}$ and $w_1 v_{2k}$ belong to $\mathcal{S}$, and writing $T_1 = x_1 u_1 w_{2k},$ $T_2 = x_2 v_1 u_{2k}$ and $T_3 = x_3 w_1 v_{2k}$, each sub-\tba~induced by $T_i$ is of type $\tbty{B}{0}$ or $\tbty{B}{1}$ with $x_i$ as the first vertex, and it is even.
\item The segments are given by
\[\begin{split}\mathcal{S} \setminus \cup_{i=1}^3 \mathcal{S}_{T_i} = \{u_{2i - 1} w_{2k + 2 - 2i}\colon i \in [k]\} &\cup \{v_{2i-1} u_{2k + 2 - 2i}\colon i \in [k]\}\\
 &\cup \{w_{2i-1} v_{2k + 2 - 2i}\colon i \in [k]\}\end{split}\]
and 
\[\begin{split}\mathcal{B} \setminus \cup_{i=1}^3 \mathcal{B}_{T_i} = \{u_{2i} w_{2k + 1 - 2i}\colon i \in [k]\} &\cup \{v_{2i} u_{2k + 1 - 2i}\colon i \in [k]\} \\
&\cup \{w_{2i} v_{2k + 1 - 2i}\colon i \in [k]\}.\end{split}\]
\item For every triple $(a,b,c) \in [2k]^3$ such that not all of $a,b,c$ are even and $a+b+c = 4k + 2$, the triple of segments $u_a w_{2k+1-a}, v_b u_{2k+1 -b}, w_c v_{2k+1-c}$ is concurrent at the point $p_{a,b,c}$. The intersections between segments in $\overline{\mathcal{S}}$ and $\mathcal{B}$, and the intersections of pairs of segments in $\overline{\mathcal{S}}$ are either on $\partial T$, or the points $p_{a,b,c}$.  
\end{enumerate}
We remark that allowing $k = 1$ in the definition of type $\trt$ would actually give type $\tbty{B}{3}$. We keep $\tbty{B}{3}$ as a basic type, as in this case the intersections between initial segments lie on $\partial T$ only, while in the type $\trt$ we insist on having at least one intersection of initial segments that is in the interior of $T$.\\

Classifying all \tbas~currently seems out of reach, however we are able to prove the following. (The assumption \textbf{(A)} below is exactly the assumption \textbf{(iii)} of Theorem~\ref{informalClassification}.) 

\begin{theorem}[A classification theorem for \tbas.]\label{classificationTheorem} Suppose that $\Delta = (T, \mathcal{S}, \mathcal{B})$ is a strong \tba, i.e. a \tba~such that the following assumption \textbf{(A)} holds.\begin{itemize}
\item[\textbf{(A)}] In every minimal $\overline{\mathcal{S}}$-region $R$, for any consecutive vertices $v_1, v_2, v_3, v_4, v_5$ appearing in this order on $\partial R$, we have that, if $l(v_1v_2)$ and $\beta(v_3)$ intersect in $T$, and $\beta(v_3)$ meets the interior of $R$, then $l(v_1v_2), \beta(v_3), l(v_4v_5)$ are concurrent.\end{itemize}
Then $\Delta$ has one of the types among $\tbty{B}{0}, \tbty{B}{1}, \tbty{B}{2}, \tbty{B}{3}, \tbty{I}{1}, \tbty{I}{2}$ and $\trt$.
\end{theorem}

\section{Deducing Theorem~\ref{blockingTheorem} from the Classification theorem}\label{sectionDeduct}

\hspace{12pt}In this section we prove Theorem~\ref{blockingTheorem}. Immediately, we move to the dual, where the theorem has the following formulation. For $n \geq 4$, we say that a pair of disjoint sets $(\mathcal{L}, \mathcal{B})$ of lines in $\mathbb{P}^2 = \mathbb{P}^2(\mathbb{R})$ is an \emph{$n$-blocking configuration} if $|\mathcal{L}| = |\mathcal{P}| + 1 = n$, no three lines in $\mathcal{L}$ are concurrent, and for every pair $l_1, l_2$ of lines in $\mathcal{L}$ there is a unique line in $\mathcal{B}$ that passes through $l_1 \cap l_2$. We refer to the lines in $\mathcal{L}$ as the \emph{initial lines}, and to the lines in $\mathcal{B}$ as the blocking lines.

\begin{theorem}\label{dualBlockingTheorem}Let $n \geq 4$ and let $(\mathcal{L}, \mathcal{B})$ be an $n$-blocking configuration. Then $n = 4$.\end{theorem}

We begin the proof by deducing some structural information about the configuration of lines in $\mathcal{L} \cup \mathcal{B}$, which will enable us to apply Theorem~\ref{classificationTheorem} and deduce Theorem~\ref{dualBlockingTheorem}.\\

By an \emph{$\mathcal{L}$-region}, we mean the closure of any connected subset of $\mathbb{P}^2 \setminus L$, where $L$ is a subset of some lines in $\mathcal{L}$. A \emph{minimal $\mathcal{L}$-region} is the closure of a connected subset of $\mathbb{P}^2 \setminus \cup \mathcal{L}$. For a $\mathcal{L}$-region $R$, we define its \emph{edges} to be the segments of lines in $\mathcal{L}$ that intersect $R$, and \emph{vertices} as the intersections of lines in $\mathcal{L}$ that lie in $R$. Finally, a vertex $v$ of $\mathcal{L}$-region is \emph{internally blocked} if the unique blocking line through $v$ meets the interior of $R$, otherwise, it is \emph{externally blocked}. 

\begin{lemma}\label{evenInternalBlocks} Let $n \geq 4$ and let $(\mathcal{L}, \mathcal{B})$ be an $n$-blocking configuration. Let $R$ be any $\mathcal{L}$-region, not necessarily minimal. Then the number of internally blocked vertices of $R$ is even.\end{lemma}

\begin{proof} We proceed by a double-counting argument. Draw $l \cap R$ for all the initial lines $l$ that meet $R$, partitioning $R$ into minimal $\mathcal{L}$-regions $R_1, R_2, \dots, R_t$. Observe that the total number $N$ of the internally blocked vertices of the regions $R_1, R_2, \dots, R_t$ can be written as $N = N_1 + N_2 + N_3$, where $N_1$ is the number of internally blocked vertices of $R$, $N_2$ is the number of internally blocked vertices which lie on the interiors of edges of $R$ and $N_3$ is the number of internally blocked vertices that lie in the interior of $R$. Our goal is to show that $N_1$ is even.

Every blocking line $b$ that crosses $R$ intersects $\partial R$ at two points, and these contribute to $N$ by 2, and $b$ also passes through some vertices in $\intr R$. But, each such vertex is blocked internally two times by $b$, for some minimal $\mathcal{L}$-regions. Hence, $N$ is even.

Observe that every initial line $l$ that meets $R$, but is not one of its edges, intersects $\partial R$ twice at interiors of edges of $R$. Thus, these two intersections contribute to $N_2$ by 2, and every such intersection is defined by a unique such initial line $l$. This shows that $N_2$ is even.

Finally, every vertex in $\intr R$ is internally blocked twice, so $N_3$ is even, hence $N_1 = N - N_2 - N_3$ is also even, as desired.\qed\end{proof}

\begin{lemma} Let $n \geq 4$ and let $(\mathcal{L}, \mathcal{B})$ be an $n$-blocking configuration. Let $R$ be any minimal $\mathcal{L}$-region. Then either all vertices of $R$ are internally blocked or all vertices of $R$ are externally blocked.\end{lemma}

\begin{proof} Suppose contrary, $R$ has two consecutive vertices $u$ and $v$, such that $u$ is internally blocked, but $v$ is externally blocked. Since a blocking line meets $\intr R$, $R$ cannot be a triangle. Let $u', v'$ be another two vertices of $R$, such that $u', u, v, v'$ are consecutive. Let $p = u'u \cap v'v$. Consider $\mathcal{L}$-region $S$ with vertices $u,v,p$. Inside $S$, among $u$ and $v$, exactly one is internally blocked vertex. Therefore, $p$ is an internally blocked vertex of $S$, with a blocking line $b \in \mathcal{B}$. But $b$ must cross the interior of $uv$, which is a contradiction with the fact that $R$ is a minimal $\mathcal{L}$-region.\qed\end{proof}

\begin{lemma} Let $n \geq 4$ and let $(\mathcal{L}, \mathcal{B})$ be an $n$-blocking configuration. Let $R$ be any minimal $\mathcal{L}$-region, with vertices $v_1, v_2, \dots, v_k$, sorted in the order as they appear on $\partial R$. Suppose that the vertices of $R$ are internally blocked. Then $k$ is even, and for every $i, j \in [k/2]$, the line $v_i v_{i + k/2} \in \mathcal{B}$, and the lines $v_{i - j - 1} v_{i-j}, v_i v_{i + k/2}, v_{i+j} v_{i + j + 1}$ are concurrent, (indices of vertices are taken modulo $k$).  \end{lemma}
\begin{proof} Let $b_i$ be the blocking line at the vertex $v_i$, and let $l_i = v_i v_{i+1} \in \mathcal{L}$. We first prove that $l_{i - j - 1}, b_i, l_{i + j}$ are concurrent by induction on $j \in \{0, 1, 2, \dots, k/2\}$. Observe that when $j = k/2$, we have that $l_{i - k/2 - 1}, b_i, l_{i + k/2}$ are concurrent. But the lines $l_{i - k/2 - 1} = v_{i - k/2 - 1} v_{i - k/2}$ and $l_{i + k/2} = v_{i + k/2}v_{i + k/2 + 1}$ already meet at $v_{i + k/2}$, which is blocked by $b_{i + k/2}$, so by uniqueness of blocking lines $b_i = b_{i + k/2} = v_i v_{i + k/2}$.\\[3pt]
\indent For the base of induction, when $j = 0$, the lines $l_{i - 1}, b_i, l_{i}$ meet at $v_i$, so the claim holds.\\
\indent Suppose that the claim holds for some $0 \leq j < k/2$, and consider $l_{i - j - 2},$ $b_i,$ $l_{i + j + 1}$. Look at the $\mathcal{L}$-region $S$ formed by lines $l_{i - j - 2}, l_{i - j - 1}, l_{i + j}, l_{i + j + 1}$. By induction hypothesis, the triple of lines $l_{i - j - 1}, b_i, l_{i+j}$ is concurrent with the common point $p_1$. Let $p_2 \colon= l_{i-j-2} \cap l_{i+j}, p_3 \colon= l_{i-j-1} \cap l_{i+j+1}$ and $p_4 \colon= l_{i-j-2} \cap l_{i + j + 1}$, so the vertices of $S$ are precisely $p_1, p_2, p_3$ and $p_4$, and our goal is to show that $p_4 \in b_i$. Look at $\mathcal{L}$-region with vertices $v_{i-j-1}, p_1, p_2$, formed by initial lines $l_{i-j-2}, l_{i-j-1}$ and $l_{i+j}$. For this region, $v_{i-j-1}$ and $p_1$ are externally blocked, so by Lemma~\ref{evenInternalBlocks}, the vertex $p_2$ must be externally blocked as well, therefore $p_2$ is internally blocked in $S$. Likewise, the vertex $p_3$ is internally blocked in $S$, so, since $p_1$ is also internally blocked, by Lemma~\ref{evenInternalBlocks}, the remaining vertex $p_4$ is internally blocked, by some blocking line $b$. But $b$ meets $R$ between vertices $v_{i - j - 1}$ and $v_{i + j}$, so it must contain some vertex $v_l$, $l \in \{i - j, i - j + 1, \dots, i + j - 1\}$. By induction hypothesis, $b_l$ meets $l_{i-j-2}$ at a point other than $p_4$, when $l < i$, and $b_l$ meets $l_{i+j + 1}$ at point other than $p_4$, for $l > i$, proving that $l = i$, as desired.\qed\end{proof}

Having acquired enough structural information about blocking configurations, we are ready to prove Theorem~\ref{dualBlockingTheorem}.

\begin{proof}[Proof of Theorem~\ref{dualBlockingTheorem}.] Observe that since no three lines in $\mathcal{L}$ are concurrent, there are $m_1, m_2, m_3 \in \mathcal{L}$ that define a region which is a minimal $\mathcal{L}$-region (this follows from the fact that any line that crosses a triangle splits that triangle into two regions, one of which is also a triangle). Denote the other $\mathcal{L}$-regions formed by $m_1, m_2$ and $m_3$ by $S_1, S_2$ and $S_3$. Applying Theorem~\ref{classificationTheorem} to the \tbas~induced by $S_1, S_2$ and $S_3$, we see that each of them is of type $\tbty{B}{1}$. Let $x_1 \colon= m_2 \cap m_3, x_2 \colon= m_3 \cap m_1, x_3 \colon= m_1 \cap m_2$ and let $e_1 \colon= m_1 \cap S, e_2 \colon= m_2 \cap S, e_3 = \colon= m_3 \cap S$. Let $a_1, a_2, \dots, a_r$ be the vertices on $m_1 \setminus e_1$, listed from $x_2$ to $x_3$, let $b_1, b_2, \dots, b_s$ be the vertices on $m_2 \setminus e_2$, listed from $x_3$ to $x_1$, and finally let $c_1, c_2, \dots, c_t$ be the vertices of $m_3 \setminus e_3$, listed from $x_1$ to $x_2$. By the definition of type $\tbty{B}{1}$, we have $r = s = t$ and $a_i b_{r + 1-i}, b_i c_{r+1-i}, c_i a_{r+1- i} \in \mathcal{L}$, for $i \in [r]$. However, $a_1$ belongs to the initial lines $a_1 b_r, a_1c_r$ and $m_1$, and $m_1$ is different from $a_1b_r$ and $a_1 c_r$, therefore $a_1b_r = a_1c_r$, making $a_1, b_r, c_r$ collinear. Similarly, $b_r, a_1, c_1$ are collinear, so $c_1, c_r, b_r, a_1$ are collinear. If $r > 1$, then $c_1 \not= c_r$, but $c_1, c_r \in m_3$, making $a_1b_r = m_3$, which is a contradiction. Therefore $r = 1$, so $n=4$, as desired.\qed\end{proof}

\section{Proof of the Classification theorem}\label{sectionClassificationProof}
\hspace{12pt}As the title suggests, this section is devoted to the proof of the Classification theorem. For an integer $n \geq 0$, we say that $n$-Classification theorem holds, if the conclusion of Theorem~\ref{classificationTheorem} holds for all \tbas~$\Delta = (T, \mathcal{S}, \mathcal{B})$ with $|\mathcal{S}| + |\mathcal{B}| \leq n$. We denote the quantity $|\mathcal{S}| + |\mathcal{B}|$ by $|\Delta|$ and call it the \emph{size} of $\Delta$. The argument will be based on induction on $|\Delta|$. Note that 0-Classification theorem holds, as $|\Delta| = 0$ implies that $\mathcal{S} = \mathcal{B} = \emptyset$, so $\Delta$ is of type $\tbty{B}{0}$.\\

Before we proceed with the proof, we need a couple of pieces of notation. Firstly, a segment is \emph{minimal} if there are no other vertices in its interior. Also recall that for a segment $xy$, we write $l(xy)$ for the line that contains the segment, and given a point $p$, write $\beta(p)$ for the blocking segment through $p$. Further, also for a segment $xy$, we write $s(xy)$ for the unique element of $\overline{\mathcal{S}}$ which contains both $x$ and $y$, if it exists. Given a $\overline{\mathcal{S}}$-region $R$, and a vertex $v$ of $R$, we say that $v$ is \emph{internally blocked} in $R$ if $\beta(v)$ passes through the interior of $R$, and otherwise we say that $v$ is \emph{externally blocked} in $R$. In particular, if there are no blocking segments through $x_i$, we say that $x_i$ is externally blocked. (By a vertex of $R$ here, we mean an intersection of initial segments that are edges of $R$, so a vertex on boundary, but in interior of an edge of $R$ is not counted as a vertex when we talk about internally or externally blocked vertices.)\\

We restate Lemma~\ref{evenInternalBlocks} here, which will be crucial to our work. As we shall be using this lemma all the time, we will not refer to it explicitly.

\begin{lemma}Suppose that $\Delta = (T, \mathcal{S}, \mathcal{B})$ is a \tba, and let $R$ be any, not necessarily minimal, $\overline{\mathcal{S}}$-region. Then, the number of internally blocked vertices of $R$ is even.\end{lemma}

\begin{proposition} \label{ForbStruct1}Let $n \geq 1$ be an integer and suppose that $(n-1)$-\ct~holds. Let $\Delta = (T, \mathcal{S}, \mathcal{B})$ be a \tba~of size $n$, and let $x_1, x_2$ and $x_3$ be the vertices of the triangle $T$. Suppose that no blocking segment passes through $x_1$. Suppose that $s_1, s_2 \in \mathcal{S}$ are two initial segments, each with the property that one of the endpoints is on the edge $x_1x_2$ and the other is on the edge $x_1x_3$. Then $s_1$ and $s_2$ are disjoint.\end{proposition}

\begin{proof}Suppose contrary, let $\Delta = (T, \mathcal{S}, \mathcal{B})$ be a \tba~of size $n$, with vertices of $T$ given by $x_1, x_2$ and $x_3$, and suppose that no blocking segment passes through $x_1$, but two initial segments $a_2 a_3$ and $b_2 b_3$ intersect at the point $c$ and $a_2,b_2 \in x_1 x_2$, $a_3,b_3 \in x_1x_3$. Without loss of generality, the vertices appear in order $x_1, b_2, a_2, x_2$ and $x_1, a_3, b_3, x_3$ on the segments $x_1x_2$ and $x_1x_3$. We consider the following cases on the positions of blocking segments.
\begin{itemize} 
\item[\textbf{Case 1}] Inside the region $a_2b_2c$, the vertices $b_2$ and $c$ are both internally blocked.
\item[\textbf{Case 2}] Inside the region $a_2b_2c$, the vertex $c$ is internally blocked, but $b_2$ is externally blocked.
\item[\textbf{Case 3}] Inside the region $a_2b_2c$, the vertex $c$ is externally blocked. Thus $c$ is internally blocked in the region $x_1 a_3 c b_2$, and by the parity of the number of internally blocked vertices, exactly one of $b_2$ and $a_3$ is internally blocked in that region, without loss of generality, $b_2$. Thus, $b_2$ is externally blocked in the region $a_2b_2c$.
\end{itemize}
We treat each case separately and we depict the steps of the proof in Figures~\ref{prop8case1Fig},~\ref{prop8case2Fig} and~\ref{prop8case3Fig}.\\

\noindent\framebox{\textbf{Case 1.}} Consider the \tba~$\Delta_1$ induced by the triangle $x_1 a_2 a_3$. Since $|\Delta_1| < |\Delta| = n$, we may apply the \ct~to $\Delta_1$. All of $x_1, a_2, a_3$ are externally blocked, so the type of $\Delta_1$ is either $\tbty{B}{i}$, for some i, or $\trt$.
\begin{figure}
\includegraphics[width = \textwidth]{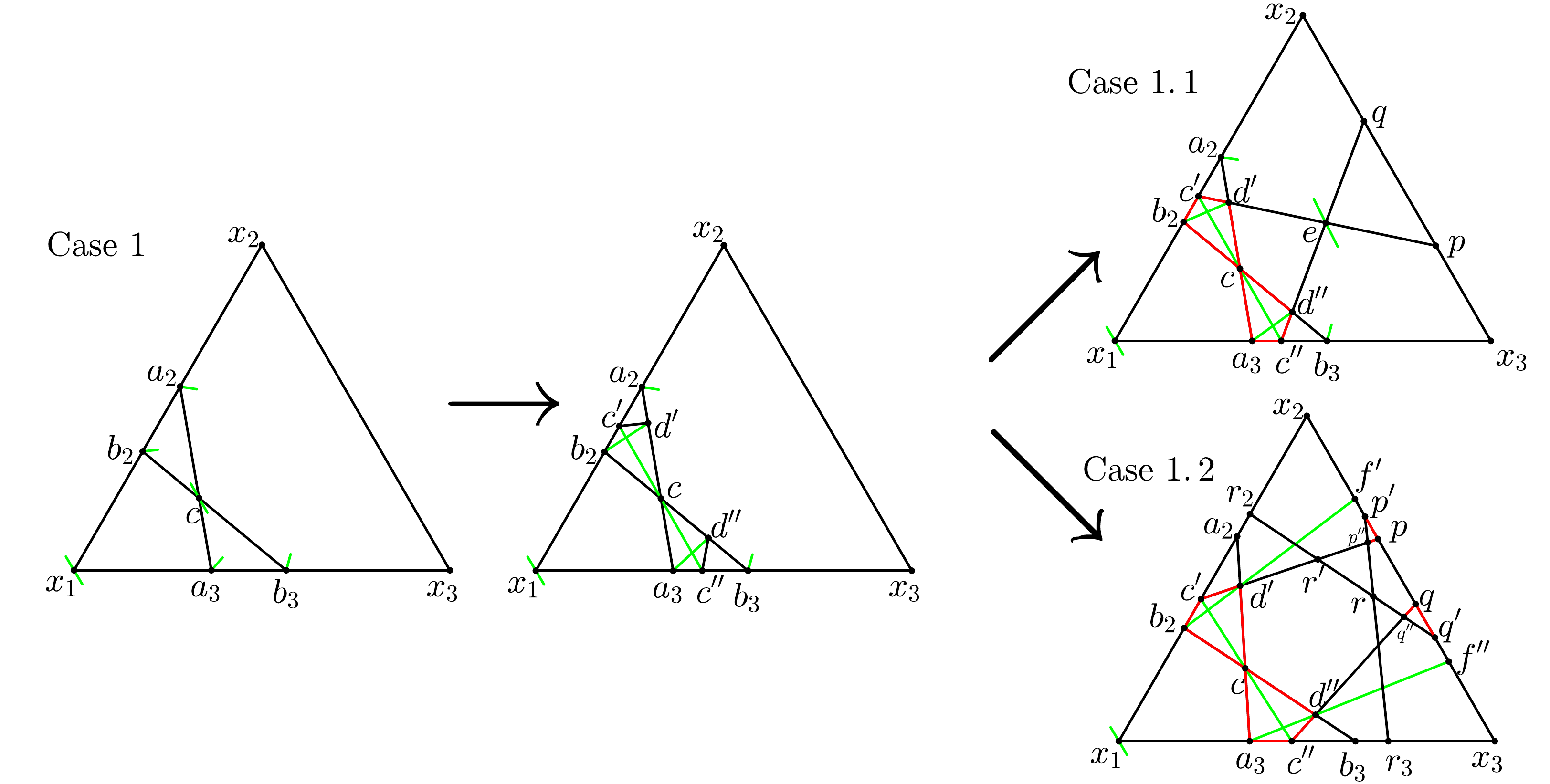}
\caption{Case 1 of proof of Proposition~\ref{ForbStruct1}}
\label{prop8case1Fig}
\end{figure}
If the type of $\Delta_1$ is $\trt$, then there is an initial segment $s \in \mathcal{S}$, whose restriction to $\Delta_1$ is $y_1y_2$, with $y_1 \in x_1a_3$ and $y_2 \in ca_3$. However, that implies that $s$ intersects segment $cb_3$, at some point $b_4$, so the \tba~$\Delta_2$, induced by the triangle $x_1b_2b_3$ has the vertex $b_3$ externally blocked, and initial segments $y_1b_4$ and $a_3c$ intersect. However, this cannot occur in any of the seven types we defined, and by the $(n-1)$-\ct,~$\Delta_2$ has one of these types, which is a contradiction.\\
Hence, the type of $\Delta_1$ is one of the four basic types, and analogously the \tba~$\Delta_2$ (induced by the triangle $x_1b_2b_3$) has also a basic type. By the definition of the basic types, if we write $d', d'', c', c''$ for the first vertices next to $c$ on $a_2c$, next to $c$ on $b_3c$, next to $b_2$ on $a_2b_2$, next to $a_3$ on $b_3a_3$, respectively, then $c'd', b_2d', cc', c''d'', a_3d''$ and $cc''$ are minimal segments, and $c'd' \in \mathcal{S}_{\Delta_1}, b_2d', cc' \in \mathcal{B}_{\Delta_1}, c''d'' \in \mathcal{S}_{\Delta_2}, a_3d'', cc'' \in \mathcal{B}_{\Delta_2}$. Furthermore, we also have that the segments $b_2c, ca_3$ are minimal. Let $p,q$ be vertices on $\partial T$ such that $c'p \colon= s(c'd')$ and $c''q\colon= s(c''d'')$.\\

\noindent\textbf{Claim.} The vertices $p$ and $q$ lie on $x_2x_3$.
\begin{proof}[Proof of the claim.] We prove the claim for $p$, the proof for $q$ is similar. Suppose contrary, $p \in x_1x_3$. Looking at \tba~$\Delta_p$ induced by $x_1 c' p$, we see that $c'$ is internally blocked, and $|\Delta_p| < n$, so the $(n-1)$-\ct~implies that $\Delta_p$ has one of types $\tbty{B}{1}, \tbty{I}{1}$ or $\tbty{I}{2}$. We can immediately discard type $\tbty{B}{1}$ as $cc'$ is blocking, but $c \notin x_1p$. On the other hand, intersecting types do not permit $\beta(c')$ crossing segment $d'a_3$ in its interior, as $a_3 \in px_1$, so we get a contradiction.\qed\end{proof}

Next, we consider the cases whether the segments $c'p$ and $c''q$ intersect or not.\\

\noindent\textbf{Case 1.1, $c'p$ and $c''q$ intersect.} Let $e \colon= c'p \cap c''q$. In the region $x_1 c'' e c'$, the vertices $c'$ and $c''$ are internally blocked, while $x_1$ is not, so $e$ is externally blocked.\\
\indent Look at the \tba~$\Delta_3$ induced by the triangle $c'' q x_3$. Its size is smaller than $n$, so the $(n-1)$-\ct~applies. As we have initial segments $d''b_3$ and $ep$, the type of $\Delta_3$ is not $\tbty{B}{0}$, nor $\tbty{B}{1}$. Suppose for a moment that is one of the two remaining basic types. Then, $\beta(e)$ crosses $x_3p$, at a point $p'$, say, and $\beta(p)$ crosses $ec''$ at a point $e'$, say, and we know that $ee', pp', e'p'$ are minimal initial segments. If $e'$ is in the interior of $ed''$, then $\beta(p)$ crosses $d''c \cup cd'$, which is impossible, as $d''c, d'c$ are minimal and $c$ is externally blocked in $cd''ed'$. But, $c''d''$ is minimal, so we must have $e' = d''$, but then $d''$ has three initial segments passing through it, $d''p', d''e, d''b_3$, which is impossible as well. Hence, $\Delta_3$ is not one of basic types.\\
\indent If the type is $\tbty{I}{1}$ or $\tbty{I}{2}$, then $q$ is internally blocked in the region $c''qx_3$. Since $e$ is externally blocked in the region $pqe$, $p$ is internally blocked in that region. Let the blocking segment through $q$ intersect $ep$ at a point $r$, say. By the definition of types $\tbty{I}{1}$ and $\tbty{I}{2}$, there is another initial segment $r_1r_2$ through $r$, with $r_1 \in px_3, r_2 \in eq$. Let $r_3 \in \partial T$ be the vertex such that $r_1r_3 = l(r_1r_2) \cap T$. Since $r_1 \in x_2x_3$ and $r_1$ and $r_3$ are on different sides of $l(c'p)$, we have $r_3 \in x_1x_2$. However, $p$ is internally blocked in the region $pqe$, and thus in the region $x_2c'p$, which thus have one of the types $\tbty{B}{1}, \tbty{I}{1}$ or $\tbty{I}{2}$, and also intersecting segments $eq$ and $r r_3$, which is a contradiction with the definitions of all these three types. Hence, $c''qx_3$ has type $\trt$.\\
\indent Since the type of $\Delta_3$ is $\trt$ and $c''d''$ is a minimal segment, it follows that $c''b_3$ and $b_3 d''$ are also minimal, and that these three segments bound a minimal $\mathcal{S}_{\Delta_3}$-region. Let $R$ be the other minimal $\mathcal{S}_{\Delta_3}$-region that contains the segment $b_3d''$, which is a hexagon by definition of $\trt$, as $\beta(d'')$ crosses $qx_3$. Let $u,v, w$ be the vertices of $R$ such that $u, b_3, d'', v, w$ are consecutive on $\partial R$. As $l(a_3d'') \cap T$ is the blocking segment through $d''$, by the assumption \textbf{(A)} $l(ub_3), l(a_3 d'')$ and $l(vw)$ are concurrent. But $l(ub_3) \cap l(a_3 d'') = a_3$, and $a_2a_3$ is another initial segment through $a_3$, which is a contradiction, as otherwise three initial segments would be concurrent.\\

\noindent\textbf{Case 1.2, $c'p$ and $c''q$ are disjoint.} Look at the \tba~$\Delta_4$ induced by the triangle $x_2 c' p$. Let $f' \colon= l(b_2d') \cap x_2x_3$ and $f'' \colon= l(a_3 d'') \cap x_2x_3$. Applying the $(n-1)$-\ct, $\Delta_4$ has one of the seven defined types. However, it cannot have a basic type, as $d' f'$ gives a blocking segment with endpoints on $x_2p$ and $c'p$, while $a_2d'$ is an initial segment and $a_2 \in c'x_2$.\\
\indent Suppose for a moment that $\Delta_4$ has type $\tbty{I}{1}$ or $\tbty{I}{2}$. By definition of these types we see that $d'f'$ does not cross any initial segment, except at its endpoints. Thus, the segments $f'd', d'c, cd'', d''c''$ are minimal. Thus, there is no initial segment that crosses $qx_3$ and $qc''$ in their interiors. But looking at the \tba~$\Delta_5$ induced by $c''qx_3$, and applying $(n-1)$-\ct, we see that $\Delta_5$ has to have one of the seven types defined, and no type satisfies the requirements that $d''b_3$ is an initial segment, $d''f''$ is a blocking segment, and there are no initial segments with a vertex on $qx_3$ and a vertex on $qc''$, which is a contradiction. Therefore, $\Delta_4$ must be of type $\trt$, and similarly, $\Delta_5$ must be of the same type.\\
\indent By definition of the type $\trt$, there are vertices $p' \in px_2, p'' \in p c'$, such that $pp',pp''$ are minimal segments and $p'p''$ is an initial segment. Similarly, there are vertices $q' \in qx_3, q'' \in q c''$ such that $qq', qq''$ are minimal segments and $q'q''$ is an initial segment. Also, recalling that $c'd'$ and $c''d''$ are minimal segments, and using the definition of $\trt$, we have that $a_2d'$ and $b_3d''$ are also minimal segments. Thus $l(p'p'')$ and $l(q'q'')$ cannot cross $a_2c \cup cb_3$, so $l(p'p'')$ crosses $x_1 x_3$, at $r_3$, say, and $l(q'q'')$ crosses $x_1x_2$, at $r_2$, say. In particular, $r_2q'$ and $r_3 p'$ intersect, at some point $r$. Also, $r_2q'$ must intersect $c'p$, at some point $r'$. However, we then have the following structure. In the \tba~induced by $x_2 r_2 q'$, the vertex $q'$ is externally blocked (by minimality of the region $qq'q''$), $p, p' \in x_2q', r, r' \in r_2q'$ and $pr', p'r$ are intersecting initial segments. Applying $(n-1)$-\ct, we obtain a contradiction, as none of the seven types has this structure. This concludes the proof in the \textbf{Case 1}.\\

\noindent\framebox{\textbf{Case 2.}} Consider the \tba~$\Delta_1$ induced by the triangle $x_1a_2a_3$. By the \ct~$\Delta_1$ has type $\tbty{B}{1}, \tbty{I}{1}$ or $\tbty{I}{2}$. Immediately, we see that the basic type is not possible here, as initial segment $b_2c$ and $\beta(a_2)$ intersect. Further, the point $e_2 \colon= b_2c \cap \beta(a_2)$, has the property that $b_2e_2$ and $e_2 c$ are both minimal segments. Let $d_2 e_3$ be the other initial segment through $e_2$, with $d_2 \in x_1a_2, e_3 \in a_2a_3$. By the definition of types $\tbty{I}{1}$ and $\tbty{I}{2}$, we have $d_2 \in a_2b_2, e_3 \in ca_3$. Let $d_3 \colon= l(d_2e_2) \cap x_1x_3$. Arguing similarly for the \tba~$\Delta_2$ induced by the triangle $x_1b_2b_3$, we obtain that the segments $b_2d_2, d_2e_2, e_2b_2, e_2c, ce_3, e_2e_3, e_3d_3, d_3a_3, a_3e_3$ are all minimal. From this, we also see that $\Delta_1$ and $\Delta_2$ must both have type $\tbty{I}{1}$. Writing $x_2' \colon= \beta(b_3) \cap x_1x_2, x'_3 \colon= \beta(a_2) \cap x_1x_3$, we also obtain that $x'_2x'_3$ is a minimal initial segment, that $x'_2b_2$ and $x'_3a_3$ are also minimal and $b_2a_3$ is a blocking segment.\\
\indent Observe $d_2c$ and $d_3c$ cannot both be blocking segments, as otherwise $d_2d_3$ would simultaneously be blocking and initial. Without loss of generality, $d_3c$ is not a blocking segment. By the definition of type $\tbty{I}{1}$, there are vertices $p,q \in d_3b_3$ and $p',q' \in cb_3$, such that $d_3p, pq, cp', p'q'$ are minimal and $p'q, pq'$ are initial segments and $p'q, pq', b_3e_3$ are concurrent. We now consider cases on the position of intersection $r$ of $l(p'q)$ and $\partial T$, other than $q$.\\[6pt]
\begin{figure}
\includegraphics[width = \textwidth]{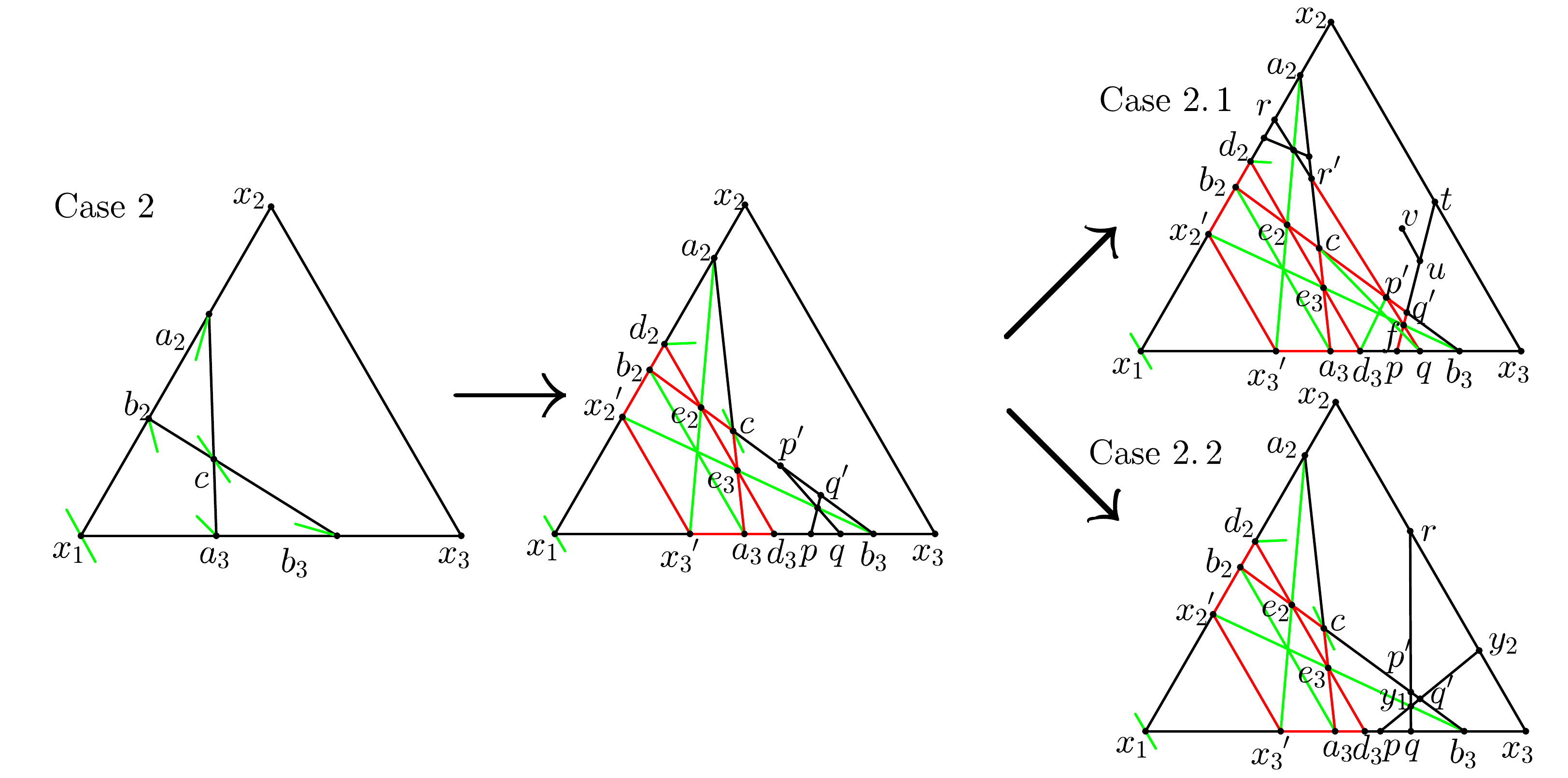}
\caption{Case 2 of proof of Proposition~\ref{ForbStruct1}}
\label{prop8case2Fig}
\end{figure}
\noindent\textbf{Case 2.1, $r \in x_1x_2$.} The type of \tba~$\Delta_3$ induced by $x_1rq$ must be $\trt$, as $x_1$ and $q$ are externally blocked, and some of the initial segments intersect. As $d_2e_2, e_2c, cp'$ are minimal segments, we must have $r \in d_2x_2$. If $r \in a_2x_2$, then we obtain a contradiction from the type $\trt$ for $\Delta_3$ and the fact that $a_2a_3$ and $d_2d_3$ intersect, while $a_2, d_2 \in x_1r, a_3d_3 \in x_1q$. Therefore $r \in d_2a_2$, and $rq$ intersects $a_2c$, at point $r'$, say. Note that the segments $q'p'$ and $p'r'$ are minimal. Returning to the \tba~$\Delta_1$ induced by $x_1a_2a_3$, which is of type $\tbty{I}{1}$, as $r$ is internally blocked in $a_2rr'$, we have vertices $s \in d_2r, s' \in a_2r'$, such that $rs, r's'$ are minimal and $rr', ss'$ are initial segments that are concurrent with $a_2e_2$.\\
\indent Extend $pq'$ to the other intersection $t$ of $l(pq')$ and $\partial T$, other than $p$. If $t \in x_1x_2$, applying the $(n-1)$-\ct~for the \tba~$\Delta_4$ induced by $x_1tp$, we have that $\Delta_4$ has type $\tbty{B}{1}, \tbty{I}{1}$ or $\tbty{I}{2}$, as $p$ is internally blocked. But, the fact that hexagonal region $x'_2 x'_3 a_3 e_3 e_2 b_2$ is minimal $\overline{\mathcal{S}_{\Delta_4}}$-region with edges on $x_1 t$ and $x_1p$ is in contradiction with the definitions of these two types. Therefore, $t \in x_2x_3$.\\
\indent Let $f \colon= pq' \cap p'q$. Applying $(n-1)$-\ct~to the \tba~induced by $x_3pt$, and observing that $q'f, qf$ are minimal segments, $q'b_3$ is initial and $fb_3$ is blocking, we must have that $qq'$ is also a blocking segment. Consider the minimal $\overline{\mathcal{S}}$-region $R$, that contains vertices $s', r', p'$ and $q'$. Let $u,v$ be the two vertices of $R$, such that $p',q', u, v$ are consecutive. As $p'r'$ and $\beta(q') = q'q$ intersect at $q$, by the assumption \textbf{(A)}, the line $l(uv)$ must also pass through $q$. However, the only initial segment other than $rq$ through $q$ is $x_1x_3$, which is disjoint from $R$, and we reach a contradiction.\\

\noindent\textbf{Case 2.2, $r \in x_2x_3$.} Applying the \ct~to the \tba~induced by $qrx_3$, we obtain a contradiction as $q$ is internally blocked in this region, while $y_1y_2 \colon= l(pq') \cap qrx_3$ and $p'b_3$ intersect, while $p', y_1 \in qr, y_2 \in rx_3, b_3 \in qx_3$, and none of the defined types has this substructure.\\

\noindent\framebox{\textbf{Case 3.}} Let $\Delta_1$ be the \tba~induced by the triangle $x_1b_2b_3$. Applying the \ct~we see that $\Delta_1$ has type $\tbty{B}{1}, \tbty{I}{1}$ or $\tbty{I}{2}$. But, $\beta(b_3)$ crosses $ca_3$, discarding $\tbty{B}{1}$ as an option. Therefore, there are collinear vertices $d_3 \in x_1a_3, e_3 \in a_3c, f_3 \in cb_3$ such that $d_3a_3, e_3d_3, e_3f_3, e_3a_3, e_3c$ and $cf_3$ are minimal, $d_3f_3$ is initial and $e_3b_3$ is blocking. Let $p \not= d_3$ be the other intersection point of $l(d_3f_3)$ and $\partial T$. If $p \in x_1x_2$, then segments $b_2b_3$ and $d_3p$ satisfy conditions of the \textbf{Case 2}, which we have proved to be impossible. Hence, $p \in x_2x_3$.\\
\begin{figure}
\includegraphics[width = \textwidth]{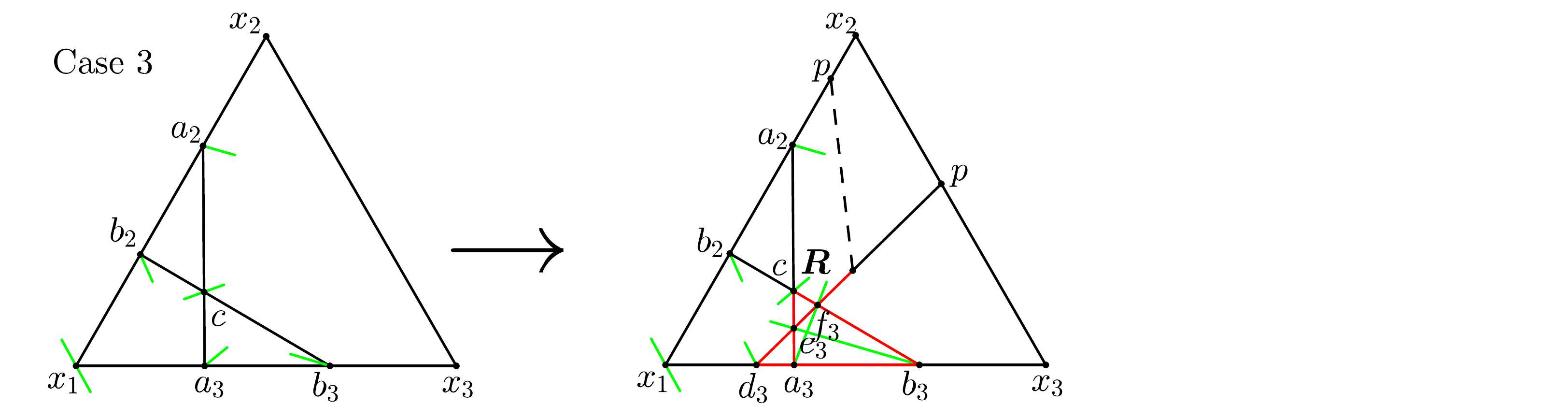}
\caption{Case 3 of proof of Proposition~\ref{ForbStruct1}}
\label{prop8case3Fig}
\end{figure}
\indent Apply the $(n-1)$-\ct~to the \tba~induced by $pd_3x_3$. Since $e_3f_3$ and $e_3a_3$ are minimal, it follows that $f_3b_3$ and $a_3b_3$ are also minimal and $a_3f_3$ is blocking. But, look at the minimal $\overline{\mathcal{S}}$-region $R$, that has $f_3$ and $c$ as two vertices, but not $e_3$. Since $l(a_2c)$ and $\beta(f_3)$ meet at $a_3$, by the assumption \textbf{(A)}, we have $a_3 \in l(s)$, where $s$ is another segment of $R$. But, the only other initial segment through $a_3$ is $x_1x_3$, which is disjoint from $R$, thus we have a contradiction.\qed\end{proof}

\begin{proposition} \label{ForbStruct2} Let $n \geq 1$ be an integer and suppose that $(n-1)$-\ct~holds. Let $\Delta = (T, \mathcal{S}, \mathcal{B})$ be a \tba~of size $n$, and let $x_1, x_2$ and $x_3$ be the vertices of the triangle $T$. Suppose there are blocking segments through $x_1$ and $x_2$. Suppose also that $a_1a_2, b_1b_2 \in \mathcal{S}$ are two initial segments, such that $a_1 \in x_1x_2, b_2 \in x_2x_3, a_2,b_1 \in x_1x_3$. Then $a_1a_2$ and $b_1b_2$ are disjoint.\end{proposition}

\begin{proof} Let $c$ be the intersection $a_1a_2 \cap b_1b_2$. Depending on the blocking segment through $c$ and $a_1$, we have the following four cases.
\begin{itemize}
\item[\textbf{Case 1}] In the region $x_2a_1cb_2$, both $a_1$ and $c$ are internally blocked.
\item[\textbf{Case 2}] In the region $x_2a_1cb_2$, $a_1$ is internally blocked, while $c$ is externally blocked.
\item[\textbf{Case 3}] In the region $x_2a_1cb_2$, $a_1$ is externally blocked, while $c$ is internally blocked.
\item[\textbf{Case 4}] In the region $x_2a_1cb_2$, both $a_1$ and $c$ are externally blocked.
\end{itemize}

We treat each case separately and we depict the steps of the proof in Figures~\ref{prop9case1Fig},~\ref{prop9case2Fig},~\ref{prop9case3Fig} and~\ref{prop9case4Fig}.\\

\noindent\framebox{\textbf{Case 1.}} We have $b_2$ externally blocked in the region $b_1b_2x_3$, and in the region $b_1ca_2$, the vertex $b_1$ is externally blocked, while $a_2$ is internally blocked. Applying the $(n-1)$-\ct~to the \tba~induced by $b_1b_2x_3$, it has a basic type or $\trt$. In either of these cases, from the definition of the types, there are vertices $d \in b_1a_2, e \in b_1c$ such that $ce, ed, da_2$ are minimal initial segments, and $cd, a_2e$ are blocking. Applying the $(n-1)$-\ct~to the \tba~induced by the triangle $x_1a_1a_2$, we see that it has type $\tbty{I}{1}$ or $\tbty{I}{2}$, so $l(de)$ must cross the segment $a_1c$, at some point $f$, and also $fc$ is minimal. Let $R$ be the minimal $\overline{\mathcal{S}}$-region with vertices $c,f$, but not $e$. Let $f', u, v$ be the vertices of $R$, such that $f', f,c,u,v$ are consecutive, appearing in this order of $\partial R$. As $l(ff') = l(de)$ and $\beta(c)$ intersect at $d$, by the assumption \textbf{(A)}, we must have $d \in l(uv)$. However, the other initial segment through $d$, apart from $de$, is $x_1x_3$, which is disjoint from $R$, and we have a contradiction in this case.\\[6pt]
\begin{figure}
\begin{center}
\includegraphics[width = 0.7\textwidth]{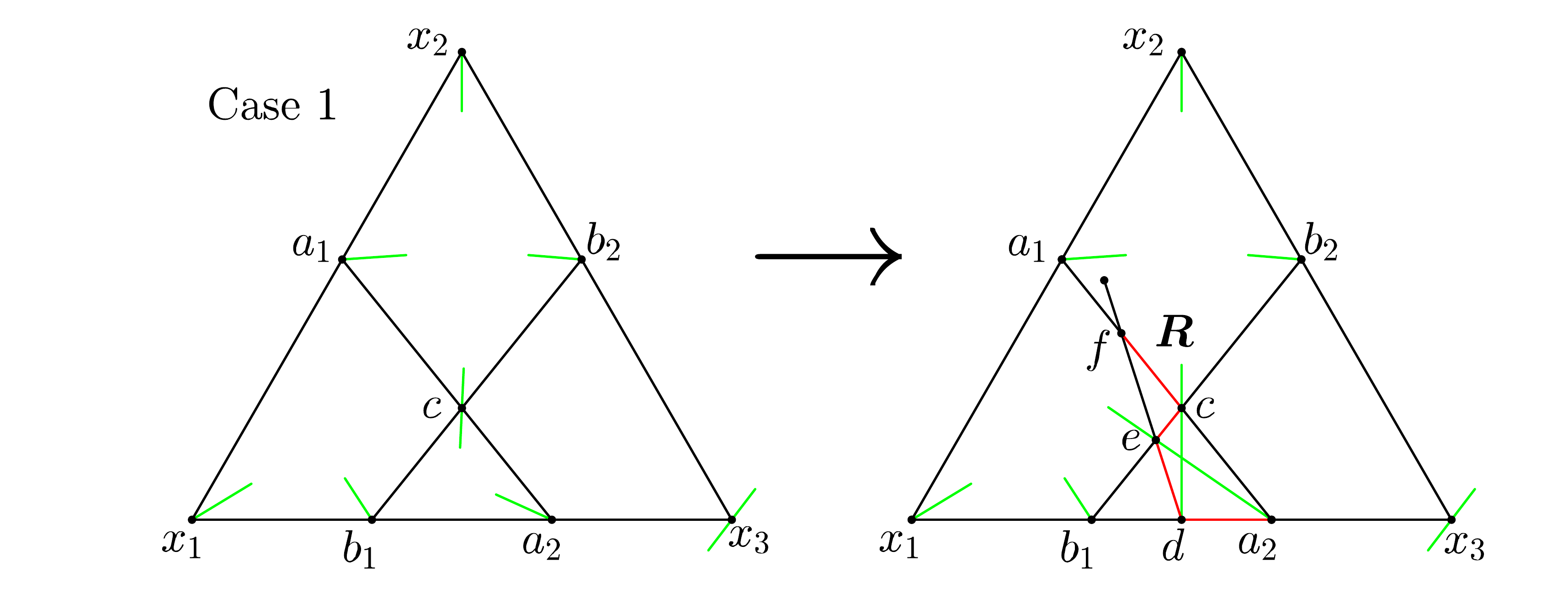}
\caption{Case 1 of proof of Proposition~\ref{ForbStruct2}}
\label{prop9case1Fig}
\end{center}\end{figure}
\noindent\framebox{\textbf{Case 2.}} We have $b_2$ internally blocked in the region $b_1b_2x_3$, and in the region $b_1ca_2$, the vertices $b_1$ and $a_2$ are internally blocked. Let $e \colon= \beta(a_2) \cap b_1c$. Applying the $(n-1)$-\ct~to the \tba~induced by $x_1a_1a_2$, we see that its type is either $\tbty{I}{1}$ or $\tbty{I}{2}$. In either case, there are vertices $d \in x_1b_1, f \in a_2c$ such that $b_1d, de, ef, fc$ are minimal initial segments. Note that $l(ef)$ intersects the segment $b_2x_3$, at some point $p$, say. However, applying $(n-1)$-\ct~to the \tba~induced by $b_1b_2x_3$, we obtain a contradiction, as no type permits a configuration where $b_1, b_2$ are internally blocked, and $ca_2, ep$ cross.\\[6pt]
\begin{figure}\begin{center}
\includegraphics[width = 0.7\textwidth]{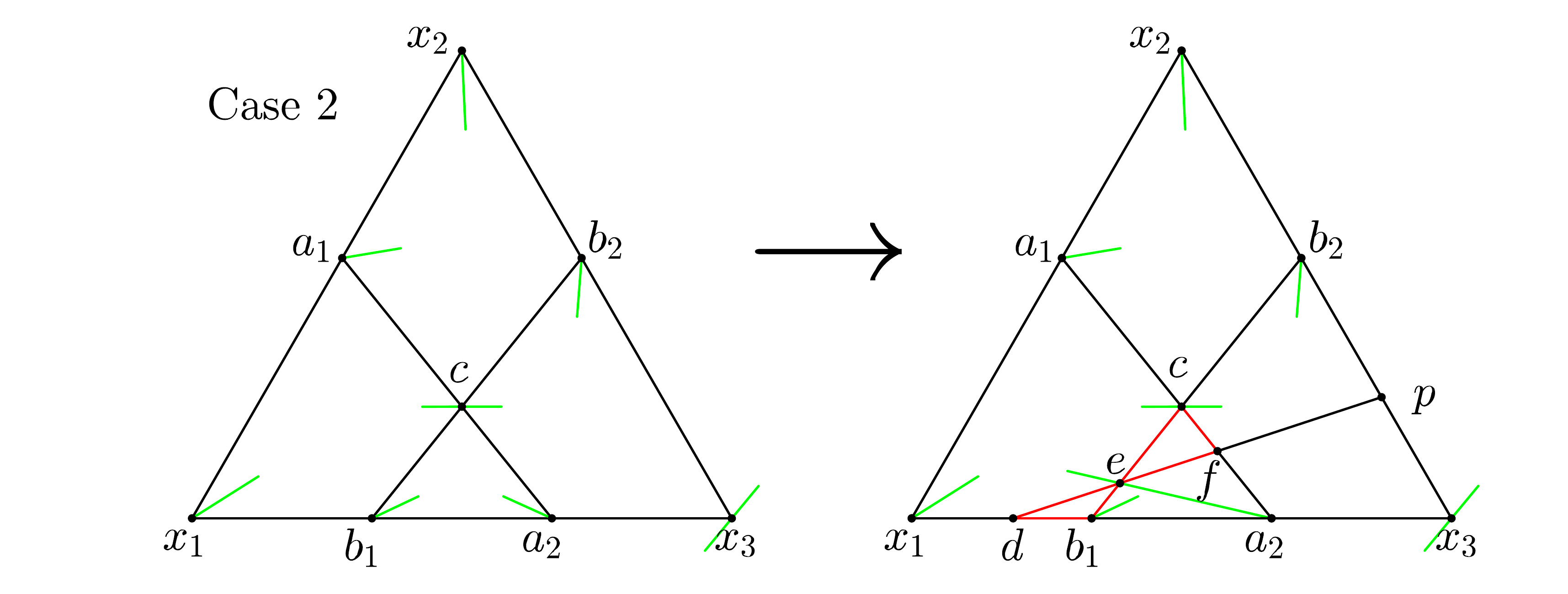}
\caption{Case 2 of proof of Proposition~\ref{ForbStruct2}}
\label{prop9case2Fig}
\end{center}\end{figure}
\noindent\framebox{\textbf{Case 3.}} We have $b_2$ internally blocked in the region $b_1b_2x_3$, and in the region $b_1ca_2$, the vertex $b_1$ is internally blocked, while $a_2$ is externally blocked. Apply the $(n-1)$-\ct~to the \tba~induced by $x_1a_1a_2$. Its type is one of $\tbty{B}{1}, \tbty{I}{1}$ and $\tbty{I}{2}$, but in any case, $b_1c$ is minimal, and there are vertices $d \in b_1a_2, e \in ca_2$ such that $ce,b_1d$ and $de$ are minimal initial segments, and $cd$ and $b_1e$ are blocking.  Apply the $(n-1)$-\ct~to the \tba~induced by $b_1b_2x_3$, which must then have type $\tbty{I}{1}$ or $\tbty{I}{2}$, and in particular $l(de)$ crosses $cb_2$, at a point $f$, say.\\
\indent Let $p \not= d$ be the other point of intersection of $l(de)$ with $\partial T$. If $p \in x_2x_3$, then the segments $pd, b_1b_2$ and the vertex $x_3$ form a configuration that is impossible by Proposition~\ref{ForbStruct1}, hence $p \in x_1x_2$. As $pd$ crosses $a_1a_2$ at $e$, we actually have $p \in a_1x_2$. However, applying $(n-1)$-\ct~to the \tba~induced by $px_1d$, we obtain a contradiction, as no type allows a subconfiguration, where $b_1f$ and $a_1e$ cross, $x_1$ and $d$ are internally blocked, and $e,f \in pd, b_1 \in x_1d, a_1 \in x_1p$.\\[6pt]
\begin{figure}\begin{center}
\includegraphics[width = 0.7\textwidth]{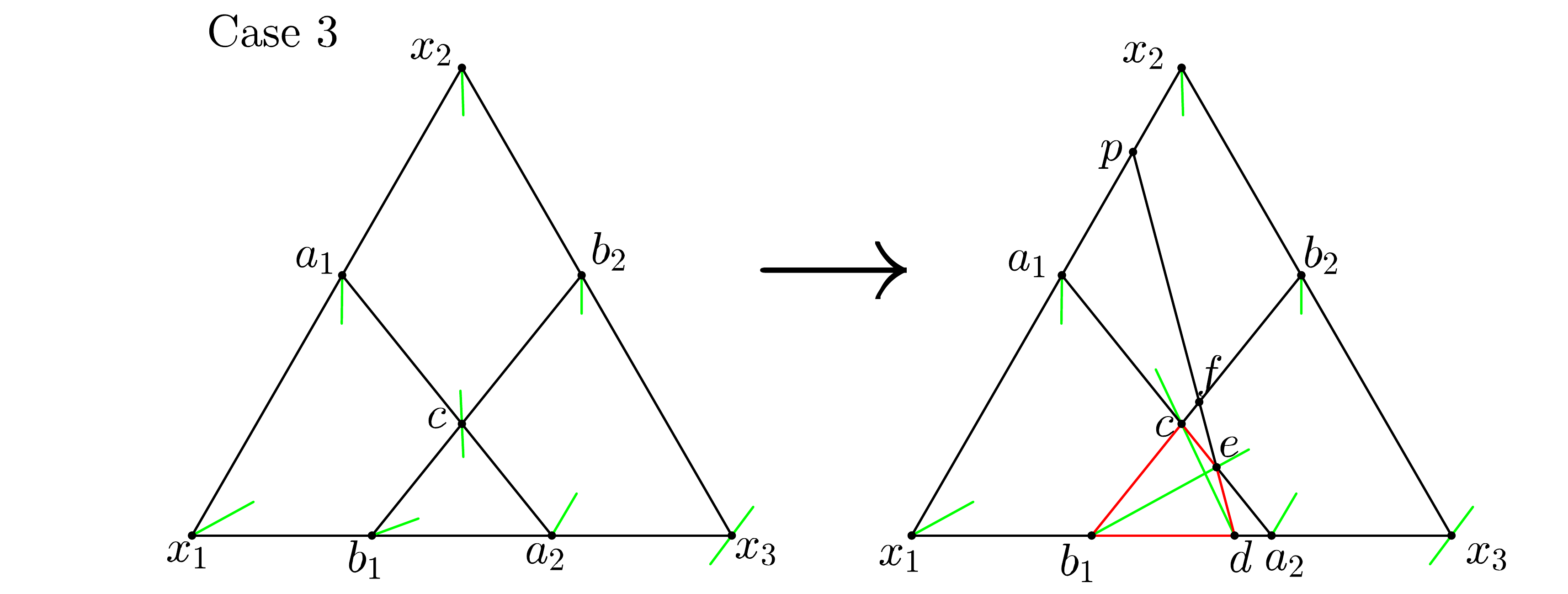}
\caption{Case 3 of proof of Proposition~\ref{ForbStruct2}}
\label{prop9case3Fig}
\end{center}\end{figure}
\noindent\framebox{\textbf{Case 4.}} We have $b_2$ externally blocked in the region $b_1b_2x_3$, and in the region $b_1ca_2$, the vertices $b_1$ and $a_2$ are externally blocked. Applying the $(n-1)$-\ct~to the \tba~$\Delta_1$ induced by $x_1a_1a_2$, we see that its type is $\tbty{B}{1}, \tbty{I}{1}$ or $\tbty{I}{2}$. In particular, the \tba~induced by $b_1a_2c$ is of type $\tbty{B}{0}$ or $\tbty{B}{1}$ and $b_1c$ is minimal. On the other hand, applying the $(n-1)$-\ct~to the \tba~$\Delta_2$ induced by $b_1b_2x_3$, which must have a basic type or $\trt$, we also see that as $b_1c$ is minimal, so are $b_1a_2$ and $a_2c$.
\begin{figure}\begin{center}
\includegraphics[width = 0.7\textwidth]{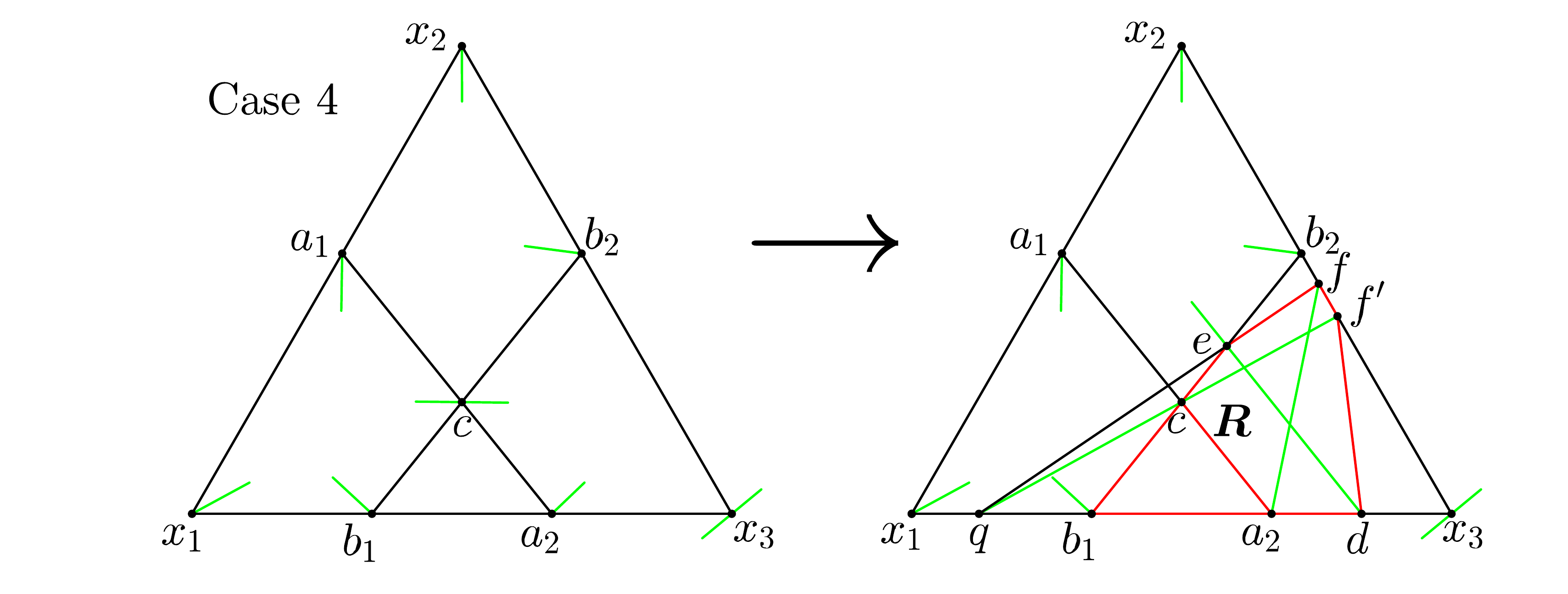}
\caption{Case 4 of proof of Proposition~\ref{ForbStruct2}}
\label{prop9case4Fig}
\end{center}\end{figure}
Assume now that there is no initial segment with one vertex on $cb_2$ and the other vertex on $a_2x_3$. Thus $\Delta_2$ has type $\trt$, and there are vertices $d \in a_2x_3, f,f' \in x_3b_2, e \in b_2c$ such that $a_2d, df', f'f, fe, ec$ are minimal initial segments, that bound a minimal $\overline{\mathcal{S}}$-region $R$, and $a_2f, de, cf'$ are blocking. Observe that by the assumption \textbf{(A)} for $R$, as $l(a_2c)$ and $l(ff')$ are disjoint in $T$, we must have $l(de)$ disjoint from these two lines in $T$, as well. Hence $l(de)$ crosses segment $a_1x_2$. Thus, $\beta(x_2)$ does not pass through $e$ and $c$, and $fe, ec, cb_1$ are minimal, so $\beta(x_2)$ has to cross $a_1c$ and either crosses $x_1b_1$ or contains $b_1$.\\
\indent Next, we show that $l(cf')$ crosses $x_1x_3$. If $\Delta_1$ has type $\tbty{B}{1}$, then this is true. If $l(cf')$ does not cross $x_1x_3$, then, $\Delta_1$ has type $\tbty{I}{1}$ or $\tbty{I}{2}$, but in that case, the only blocking segments that could cross $l(cf')$ in $\Delta_1$ are $\beta(x_1), \beta(a_1)$ and $\beta(b_1)$. Thus $\beta(x_2)$ passes through $b_1$. But $\beta(b_1)$ crosses the interior of $a_1c$, but types $\tbty{I}{1}$ and $\tbty{I}{2}$ imply that $\beta(b_1)$ crosses $a_1x_1$, which is a contradiction.\\
\indent Hence, $l(cf')$ crosses $x_1x_3$. Then, by the assumption \textbf{(A)} for region $R$, $l(ef), l(cf'), l(a_2d)$ are concurrent at a point $q \not= x_2$. However, Proposition~\ref{ForbStruct1} gives a contradiction, when applied to the vertex $x_3$ and the segments $fq$ and $b_1b_2$.\\[3pt]
\indent Finally, we assume that there is an initial segment $de$ with $e \in cb_2$ and $d \in a_2x_3$. To obtain a contradiction, we consider the following three cases on the position of $p \not= d$, the other intersection of $l(de)$ with $\partial T$.\\[3pt]
\noindent\textbf{Case 4.1. Suppose that $p \in x_1a_1$.} Applying $(n-1)$-\ct~to the \tba~induced by $x_1pd$, we see that $d$ must be internally blocked in this region. It follows that $e$ is internally blocked in the region $ca_2de$. However, $pd$ and $b_1b_2$ satisfy the conditions of the \textbf{Case 1}, which is impossible.\\[3pt]
\noindent\textbf{Case 4.2. Suppose that $p \in a_1x_2$.} Applying $(n-1)$-\ct~to the \tba~induced by $x_1pd$ results in a contradiction as $a_1a_2$ and $b_1e$ cross, but $x_1$ is internally blocked.\\[3pt]
\noindent\textbf{Case 4.3. Suppose that $p \in x_2x_3$.} We actually have $p \in x_2b_2$ and Proposition~\ref{ForbStruct1} says that configuration where $x_3$ is externally blocked, and $b_1b_2$ and $dp$ cross is impossible.\qed\end{proof}

\begin{proposition} \label{ForbStruct3} Let $n \geq 1$ be an integer and suppose that $(n-1)$-\ct~holds. Let $\Delta = (T, \mathcal{S}, \mathcal{B})$ be a \tba~of size $n$, and let $x_1, x_2$ and $x_3$ be the vertices of the triangle $T$. Suppose there are blocking segments through $x_1$ and $x_3$. Suppose also that $a_1a_2, b_1b_2 \in \mathcal{S}$ are two initial segments, such that $a_1 \in x_1x_2, b_2 \in x_2x_3, a_2,b_1 \in x_1x_3$. Then $a_1a_2$ and $b_1b_2$ are disjoint.\end{proposition}

\begin{proof} Suppose contrary, let $c\colon= a_1a_2 \cap b_1b_2$. Depending on the blocking segments through $c$ and $b_1$, up to symmetry, we have the following three cases.
\begin{itemize}
\item[\textbf{Case 1}] In the region $cb_1a_2$, $c$ is externally blocked, while $b_1$ is internally blocked.
\item[\textbf{Case 2}] In the region $cb_1a_2$, both $b_1$ and $c$ are externally blocked.
\item[\textbf{Case 3}] In the region $cb_1a_2$, both $b_1$ and $c$ are internally blocked.
\end{itemize}

As before, we treat each case separately and we depict the steps of the proof in Figures~\ref{prop10case1Fig},~\ref{prop10case2Fig} and~\ref{prop10case3Fig}.\\

\noindent\framebox{\textbf{Case 1.}} Note that the vertex $a_2$ is internally blocked in the region $cb_1a_2$, and that $a_1$ and $b_2$ are internally blocked in the region $x_2a_1cb_2$. Let $d\colon= \beta(a_2) \cap cb_1, e\colon= \beta(b_1) \cap ca_2$. Applying the $(n-1)$-\ct~to the \tba~$\Delta_1$ induced by $a_1a_2x_1$, we see that $\Delta_1$ has type among $\tbty{I}{1}$ and $\tbty{I}{2}$. In either case, $dc$ and $db_1$ are minimal segments, and no initial segment may cross the interior of the segment $b_1e$. Similarly, by looking at the \tba~$\Delta_2$ induced by $b_1b_2x_3$, we have that $ce, ea_2$ are minimal, and no initial segment crosses the interior of $a_2d$. It follows that the initial segment through $d$, which is different from $cb_1$, must be $de$, so $de$ is an initial segment.\\
\indent However, the type of $\Delta_1$ implies that $l(de)$ crosses $x_1b_1$, while the type of $\Delta_2$ implies that $l(de)$ crosses $a_2x_3$, so $l(de)$ crosses segment $x_1x_3$ twice, which is impossible.\\[6pt]
\begin{figure}\begin{center}
\includegraphics[width = 0.7\textwidth]{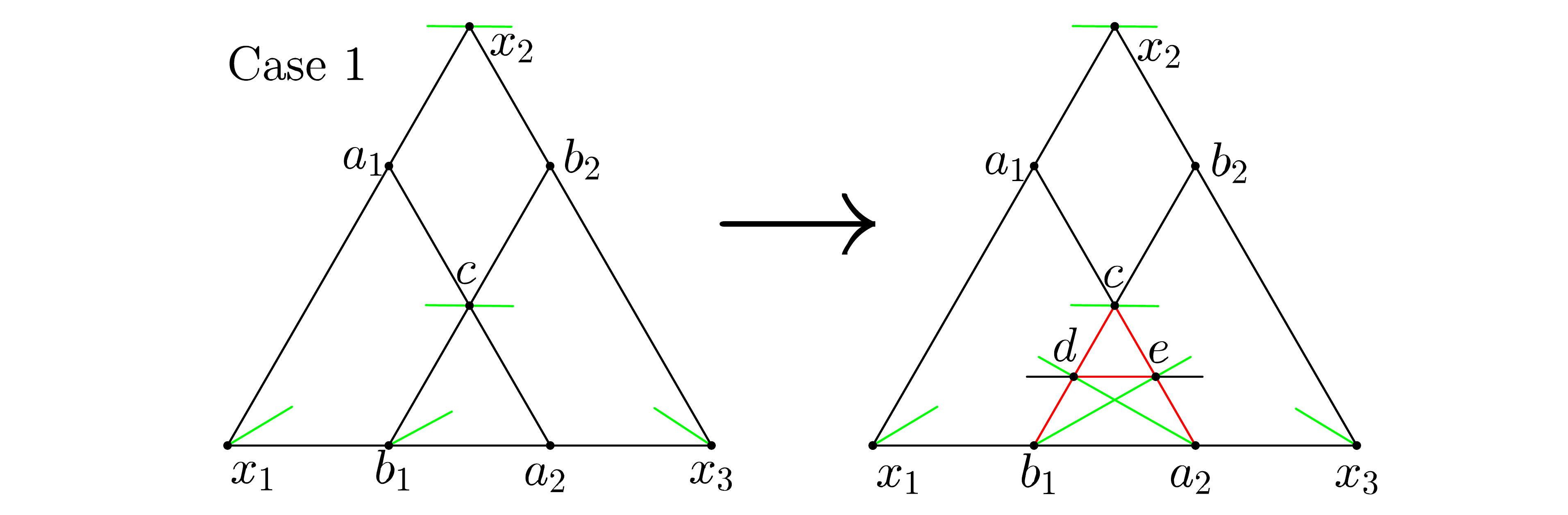}
\caption{Case 1 of Proposition~\ref{ForbStruct3}}
\label{prop10case1Fig}
\end{center}
\end{figure}
\noindent\framebox{\textbf{Case 2.}}Note that the vertex $a_2$ is externally blocked in the region $cb_1a_2$, and that $a_1$ and $b_2$ are externally blocked in the region $x_2a_1cb_2$. Applying the $(n-1)$-\ct~to the \tba~$\Delta_1$ induced by $a_1a_2x_1$, the type of $\Delta_1$ is among $\tbty{B}{1}, \tbty{I}{1}$ or $\tbty{I}{2}$. In either case, the segment $b_1c$ is minimal. Similarly, looking at the \tba~$\Delta_2$ induced by $b_1b_2x_3$, the segment $ca_2$ is also minimal.\\
\indent Suppose for a moment that $\beta(c)$ crosses $a_2x_3$, at a point $e$. Let $R$ be the minimal $\overline{\mathcal{S}}$-region with vertices $b_1$ and $c$, but not $a_2$. Let $u, v, w$ be the vertices of $R$, such that $u, b_1, c, v, w$ are consecutive and appear in that order on $\partial R$, thus $wv, vc$ are minimal. As $l(ub_1)$ and $\beta(c)$ intersect at $e$, by the assumption \textbf{(A)}, $d \in l(vw)$. Also $l(vw)$ crosses $cb_2$, let $d$ be their intersection point. From the $(n-1)$-\ct~applied to $\Delta_2$, we see that $cd, de, a_2e$ are minimal and $a_2d$ is blocking. As $cd, cv$ are minimal, so is $vd$. But, if we look at the minimal $\overline{\mathcal{S}}$-region $R'$, with vertices $v,d$, but not $c$, since $l(vc) = l(s_1)$ and $\beta(d)$ intersect at $a_2$ (where $s_1$ is the segment of $\partial R'$ through $v$, different from $vd$), it follows by the assumption \textbf{(A)}, that $a_2 \in l(s_2)$ for another segment $s_2$ of $\partial R'$. However, $a_2 \in x_1x_3$, and $x_1x_3$ is disjoint from $R'$, which is a contradiction. Therefore, $\beta(c)$ is disjoint from $a_2x_3$, and by symmetry $\beta(c)$ is also disjoint from $x_1b_1$.\\
\indent Thus $\beta(c)$ crosses $a_1x_1$, at some point $d$, and crosses $b_2x_3$, at some point $e$. Further, $\beta(x_1)$ crosses $a_1c$, at a point $f$, and $\beta(x_3)$ crosses $b_2c$, at a point $f'$. From the types of $\Delta_1$ and $\Delta_2$, we also have that $fc, f'c$ are minimal, and there are points $p \in a_1d$, and $p' \in b_2e$, such that $pf$ and $p'f'$ are initial segments. Further, also from the types of $\Delta_1$ and $\Delta_2$, we have that if an initial segment $s$ crosses $fx_1$, then $s$ must have one vertex on $a_1x_1$ and the other on $x_1b_1$, and we have that if an initial segment $s$ crosses $x_3f'$, then $s$ has one vertex on $b_2x_3$, and the other on $a_2x_3$. It follows that $l(pf) \cap \operatorname{int} f'x_3 = l(p'f') \cap \operatorname{int} fx_1 = \emptyset$.\\
\indent Observe that if it happens that $l(pf) = l(p'f')$, then Proposition~\ref{ForbStruct2}~applies to segments $pp', a_1a_2$ inside $x_1x_2x_3$ to give a contradiction. Therefore, $l(pf)$ crosses $x_2x_3$ at some point $q \not= p'$ (and $x_1x_2$ at $p$) and $l(p'f')$ crosses $x_1x_2$ at some point $q' \not= p$ (and $x_2x_3$ at $p'$). Finally, as $cf$ and $cf'$ are minimal, $pq$ and $p'q'$ must cross, but then the segments $pq, p'q'$ and the vertex $x_2$ are in a contradiction with Proposition~\ref{ForbStruct1}.\\[6pt]
\begin{figure}\begin{center}
\includegraphics[width = 0.7\textwidth]{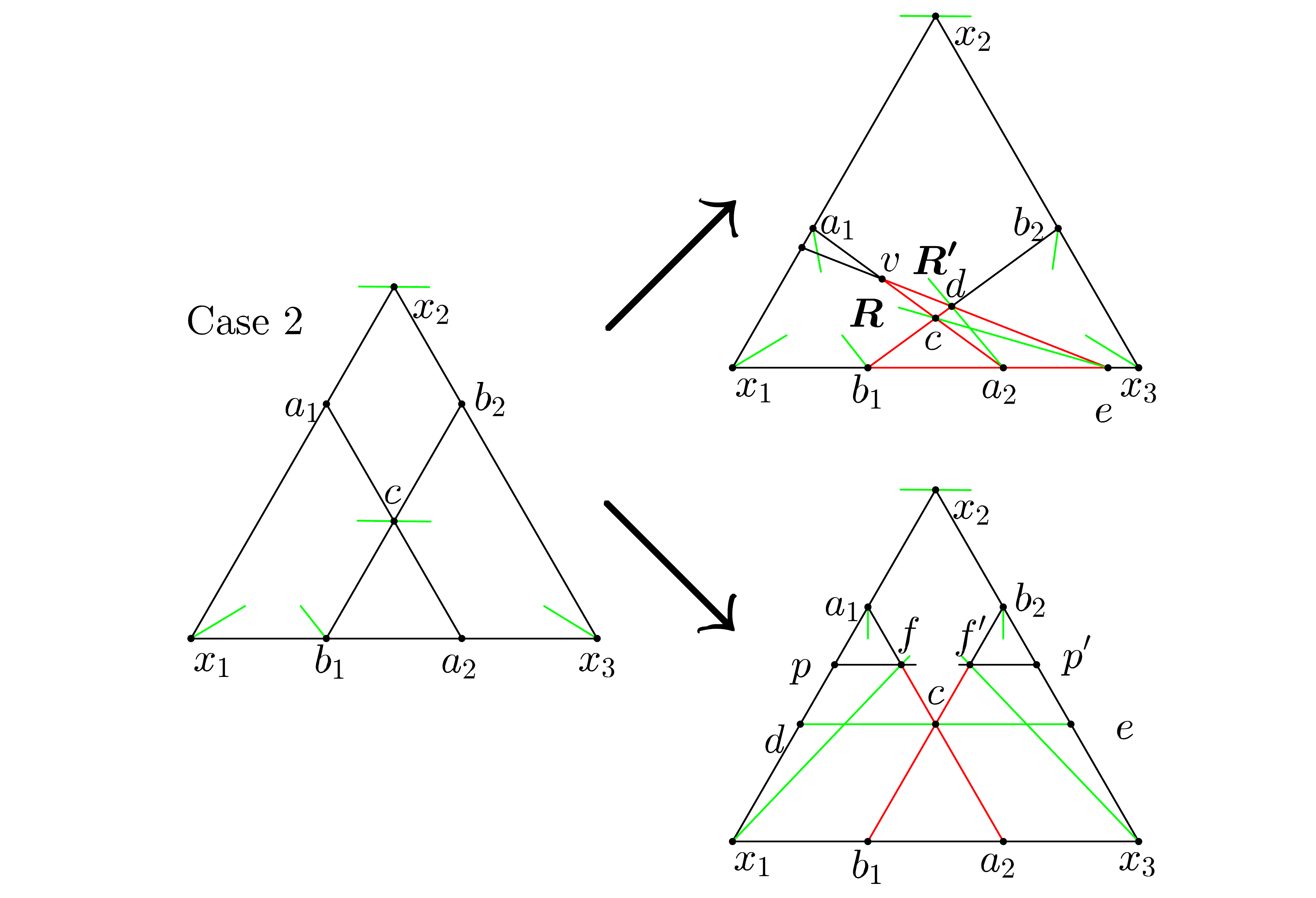}
\caption{Case 2 of Proposition~\ref{ForbStruct3}}
\label{prop10case2Fig}
\end{center}
\end{figure}
\noindent\framebox{\textbf{Case 3.}} Note that the vertex $a_2$ is externally blocked in the region $cb_1a_2$, and that $a_1$ is externally and $b_2$ is internally blocked in the region $x_2a_1cb_2$. Applying the $(n-1)$-\ct~to the \tba~$\Delta_1$ induced by $a_1a_2x_1$, we have that the type of $\Delta_1$ is among $\tbty{B}{1}, \tbty{I}{1}$ and $\tbty{I}{2}$. In either case, $b_1c$ is a minimal segment. Next, applying the $(n-1)$-\ct~to the \tba~$\Delta_2$ induced by $b_1b_2x_3$, the type of $\Delta_2$ is either $\tbty{I}{1}$ or $\tbty{I}{2}$. Since $b_1c$ is minimal, setting $d \colon= \beta(c) \cap b_1a_2, e\colon= \beta(b_1) \cap ca_2$, we must have that $de$ is an initial segment, $l(de)$ crosses $cb_2$, at some point $f$, and $cf, ce, fe, de, ea_2,da_2$ are minimal. Let $R$ be a minimal $\overline{\mathcal{S}}$-region such that $c,f \in \partial R$, but $e \notin \partial R$. Let $u, v, w$ be the vertices of $R$, such that $u, f, c, v, w$ are consecutive and appear in this order on $\partial R$. Since $\beta(c)$ and $l(uf) = l(fe)$ meet at $d$, by the assumption \textbf{(A)}, it follows that $d \in l(vw)$. However, the other initial segment through $d$, apart from $de$, is $x_1x_3$, which is disjoint from $R$, which is a contradiction.
\begin{figure}\begin{center}
\includegraphics[width = 0.7\textwidth]{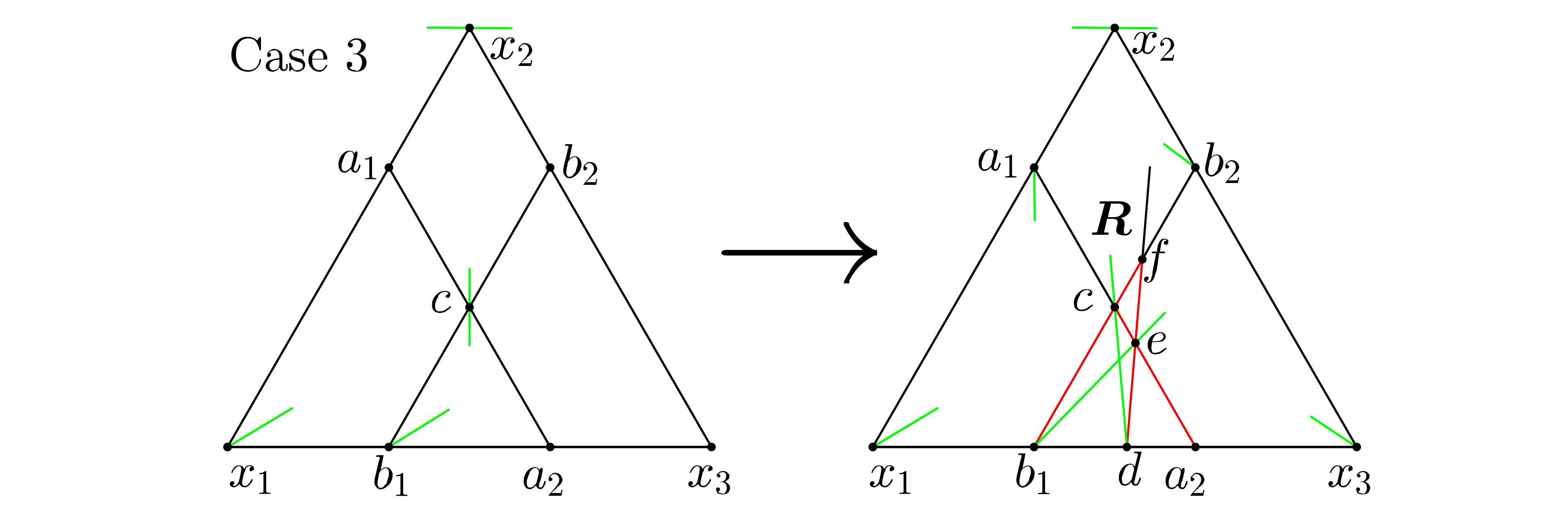}
\caption{Case 3 of Proposition~\ref{ForbStruct3}}
\label{prop10case3Fig}
\end{center}
\end{figure}\qed\end{proof}

\begin{lemma}\label{triangularBlocksLemma} Let $n \geq 1$ be an integer and suppose that $(n-1)$-\ct~holds. Let $\Delta = (T, \mathcal{S}, \mathcal{B})$ be a \tba~of size $n$, and let $x_1, x_2$ and $x_3$ be the vertices of the triangle $T$. Suppose that there are vertices $a_1 \in x_1x_2, a_2,b_1 \in x_1x_3, b_2 \in x_2x_3$ such that $a_1a_2$ and $b_1b_2$ are initial segments, intersecting at a point $c$. Then, in the region $a_1cb_2x_2$, the vertices $c, x_2$ are externally blocked and $a_1, b_2$ are internally blocked. Also, $x_1, x_3$ are externally blocked, and in the region $cb_1a_2$, the vertices $a_2,b_1$ are externally blocked.\end{lemma}

\begin{proof} By Propositions~\ref{ForbStruct2} and~\ref{ForbStruct3}, we have $x_1, x_2, x_3$ all externally blocked. Looking at the blocking segments $\beta(c)$ and $\beta(a_2)$, we have the following three cases, up to symmetry.
\begin{itemize}
\item[\textbf{Case 1}] In the region $a_1cb_2x_2$, the vertices $a_1, c, b_2, x_2$ are externally blocked. In the region $cb_1a_2$, the vertices $a_2,b_1$ are internally blocked.
\item[\textbf{Case 2}] In the region $a_1cb_2x_2$, the vertices $a_1, c$ are internally blocked and $b_2, x_2$ are externally blocked. In the region $cb_1a_2$, the vertex $b_1$ is internally blocked and $a_2$ is externally blocked.
\item[\textbf{Case 3}] In the region $a_1cb_2x_2$, the vertices $a_1, b_2$ are internally blocked and $c, x_2$ are externally blocked. In the region $cb_1a_2$, the vertices $a_2,b_1$ are externally blocked.
\end{itemize}
Observe that the \textbf{Case 3} is exactly the conclusion of the lemma, so we just need to discard the first two cases. As before, we depict the steps in the proof in Figure~\ref{triangleBlocksFig}.\\[6pt]
\begin{figure}
\begin{center}
\includegraphics[width = 0.7\textwidth]{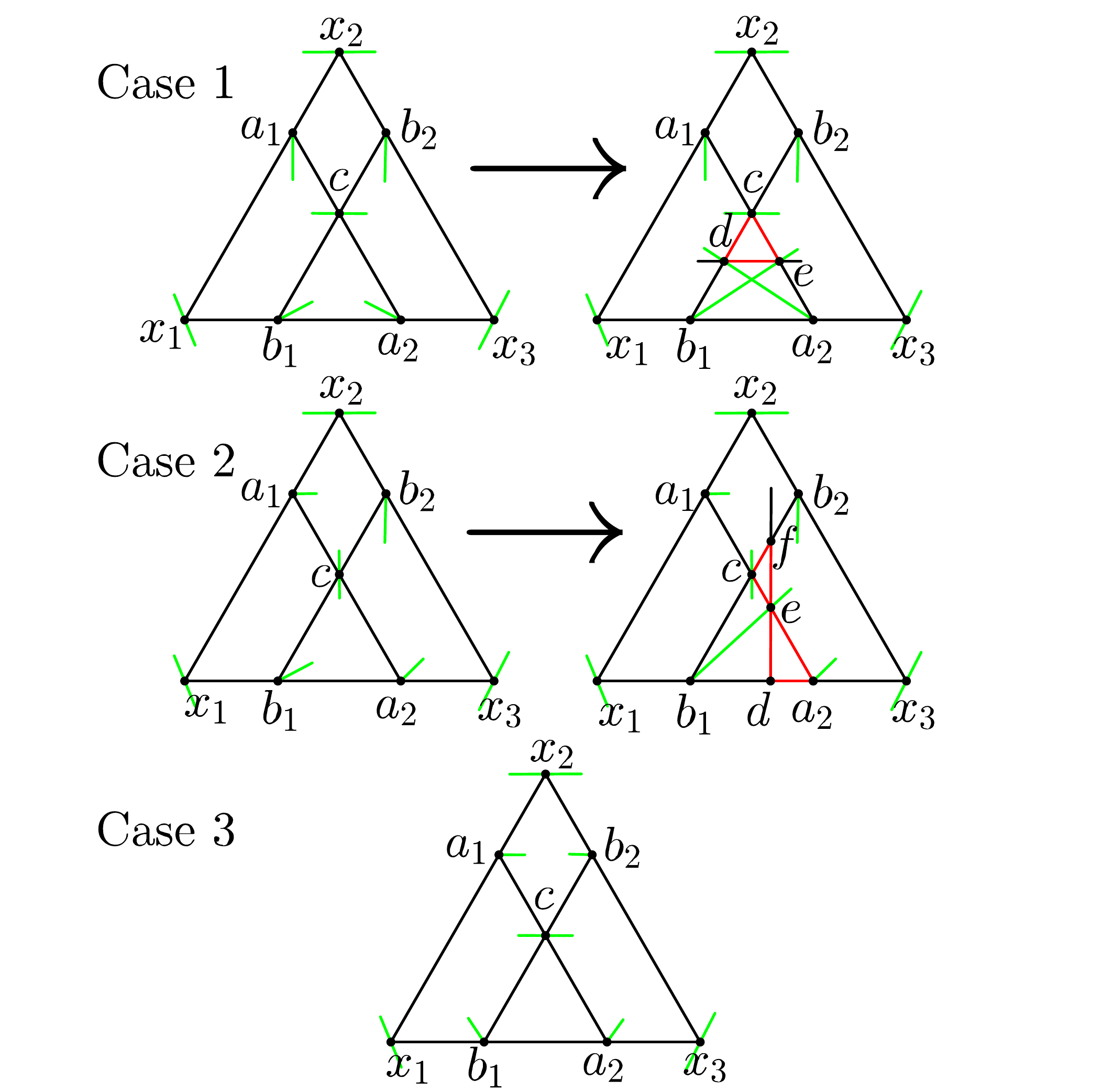}
\label{triangleBlocksFig}
\caption{Cases in the proof of Lemma~\ref{triangularBlocksLemma}}
\end{center}
\end{figure}
\noindent\framebox{\textbf{Case 1.}} Let $d \colon= \beta(a_2) \cap b_1c, e \colon= \beta(b_1) \cap a_2c$. By the $(n-1)$-\ct~applied to the \tba~$\Delta_1$ induced by $a_1a_2x_1$, the segment $cd$ is  minimal, and there is another initial segment through $d$ that crosses $ce$, possibly through $e$. Similarly, by applying the $(n-1)$-\ct~to the \tba~$\Delta_2$ induced by $b_1b_2x_3$, $ce$ is also minimal segment, so it follows that $de$ is itself a minimal initial segment. But, as $\Delta_1$ has type $\tbty{I}{1}$ or $\tbty{I}{2}$, it follows that $l(de)$ crosses $x_1b_1$. Similarly, from the type of $\Delta_2$, $l(de)$ crosses $a_2x_3$, hence $l(de)$ meets $x_1x_3$ twice, which is impossible.\\

\noindent\framebox{\textbf{Case 2.}} Let $e \colon= \beta(b_1) \cap ca_2$. By the $(n-1)$-\ct~applied to the \tba~$\Delta_1$ induced by $b_1b_2x_3$, $\Delta_1$ has type $\tbty{I}{1}$ or $\tbty{I}{2}$, and in particular, there are vertices $d \in b_1a_2,f \in cb_2$, such that $d,e,f$ lie on the same initial segment, and $ce,ef,fc,ed,da_2,a_2e$ are minimal. Let $p \not= d$, be the other intersection of the line $l(de)$ with $\partial T$. If $p \in x_2x_3$, then we have a substructure where $b_1b_2, dp$ intersect, and $x_3$ is externally blocked, which is forbidden by Proposition~\ref{ForbStruct1}, yielding a contradiction. Therefore, $p \in x_1x_2$. Once again, Proposition~\ref{ForbStruct1} gives contradiction, as $a_1a_2, dp$ intersect and $x_1$ is externally blocked.\qed\end{proof}

\begin{proposition}\label{triangleCoversProp}Let $n \geq 1$ be an integer and suppose that $(n-1)$-\ct~holds. Let $\Delta = (T, \mathcal{S}, \mathcal{B})$ be a \tba~of size $n$, and let $x_1, x_2$ and $x_3$ be the vertices of the triangle $T$. Suppose that there are vertices $a_3, b_3 \in x_1x_2$, appearing in order $x_1, b_3, a_3, x_2$, $c_2, a_2 \in x_1x_3$, appearing in order $x_1, c_2, a_2, x_3$, and $b_1, c_1 \in x_3x_2$, appearing in order $x_3, b_1, c_1, x_2$, such that $a_2a_3, b_1b_3, c_1c_2 \in \mathcal{S}$. Let $d_1 = b_1b_3\cap c_1c_2, d_2 = a_2a_3 \cap c_1c_2, d_3 = a_2a_3\cap b_1b_3$. Suppose additionally that $d_1$ is in the interior of $x_1a_2a_3$. Then $\Delta$ has type $\trt$.\end{proposition}

\begin{proof} Apply the $(n-1)$-\ct~to the \tba~$\Delta_a$ induced by $x_1a_3a_2$. As $b_1b_3 \cap \Delta_a$ and $c_1c_2 \cap \Delta_a$ intersect, it must have type $\trt$. By definition, there are an integer $k \geq 2$ and vertices $u^{(a)}_1, u^{(a)}_2, \dots,$ $u^{(a)}_{2k}$ appearing on $x_1a_3$, in that order from $x_1$ to $a_3$, there are vertices $v^{(a)}_1, v^{(a)}_2, \dots,$ $v^{(a)}_{2k}$ appearing on $a_3a_2$, in that order from $a_3$ to $a_2$, and there are vertices $w^{(a)}_1, w^{(a)}_2, \dots,$ $w^{(a)}_{2k}$ appearing on $a_2x_1$, in that order from $a_2$ to $x_1$, such that segments $u^{(a)}_iu^{(a)}_{i+1}, v^{(a)}_iv^{(a)}_{i+1},$ $w^{(a)}_iw^{(a)}_{i+1}$ are minimal, and the \tbas~induced by $x_1u^{(a)}_1w^{(a)}_{2k}, a_3v^{(a)}_1u^{(a)}_{2k}$ and $a_2w^{(a)}_1v^{(a)}_{2k}$, are of type $\tbty{B}{0}$ or $\tbty{B}{1}$, with minimal initial segments $u^{(a)}_1w^{(a)}_{2k}, v^{(a)}_1u^{(a)}_{2k}$ and $w^{(a)}_1v^{(a)}_{2k}$. The initial segments, outside the three small regions $x_1u^{(a)}_1w^{(a)}_{2k},$ $a_3v^{(a)}_1u^{(a)}_{2k},$ $a_2w^{(a)}_1v^{(a)}_{2k}$ are given by $u^{(a)}_{2i-1} w^{(a)}_{2k+2-2i}, v^{(a)}_{2i-1} u^{(a)}_{2k+2-2i},$ $w^{(a)}_{2i-1} v^{(a)}_{2k+2-2i}$, for $i = 1, 2, \dots, k$, and the blocking segments outside the three small regions are given by $w^{(a)}_{2i-1} u^{(a)}_{2k+2-2i},$ $u^{(a)}_{2i-1} v^{(a)}_{2k+2-2i},$ $v^{(a)}_{2i-1} w^{(a)}_{2k+2-2i}$, for $i = 1, 2, \dots, k$.\\
\indent Similarly, apply the $(n-1)$-\ct~to the \tba~$\Delta_b$ induced by $x_2b_1b_3$. It must have type $\trt$. By definition, there are an integer $l \geq 2$ and vertices $u^{(b)}_1, u^{(b)}_2, \dots, u^{(b)}_{2l}$ appearing on $b_3x_2$, in that order from $b_3$ to $x_2$, there are vertices $v^{(b)}_1, v^{(b)}_2, \dots, v^{(b)}_{2l}$ appearing on $x_2b_1$, in that order from $x_2$ to $b_1$, and there are vertices $w^{(b)}_1, w^{(b)}_2, \dots, w^{(b)}_{2l}$ appearing on $b_1b_3$, in that order from $b_1$ to $b_3$, such that segments $u^{(b)}_iu^{(b)}_{i+1},$ $v^{(b)}_iv^{(b)}_{i+1},$ $w^{(b)}_iw^{(b)}_{i+1}$ are minimal, and the \tbas~induced by $b_3u^{(b)}_1w^{(b)}_{2l},$ $x_2v^{(b)}_1u^{(b)}_{2l}$ and $b_1w^{(b)}_1v^{(b)}_{2l}$, are of type $\tbty{B}{0}$ or $\tbty{B}{1}$, with minimal initial segments $u^{(b)}_1w^{(b)}_{2l},$ $v^{(b)}_1u^{(b)}_{2l},$ $w^{(b)}_1v^{(b)}_{2l}$. The initial segments, outside the three small regions $b_3u^{(b)}_1w^{(b)}_{2l},$ $x_2v^{(b)}_1u^{(b)}_{2l},$ $b_1w^{(b)}_1v^{(b)}_{2l}$ are given by $u^{(b)}_{2i-1} w^{(b)}_{2l+2-2i},$ $v^{(b)}_{2i-1} u^{(b)}_{2l+2-2i},$ $w^{(b)}_{2i-1} v^{(b)}_{2l+2-2i}$, for $i = 1, 2, \dots, l$, and the blocking segments outside the three small regions are given by $w^{(b)}_{2i-1} u^{(b)}_{2l+2-2i},$ $u^{(b)}_{2i-1} v^{(b)}_{2l+2-2i},$ $v^{(b)}_{2i-1} w^{(b)}_{2l+2-2i}$, for $i = 1, 2, \dots, l$.\\
\indent Apply the $(n-1)$-\ct~to the \tba~$\Delta_c$ induced by $x_3c_1c_2$. It must have type $\trt$, so by definition, there are and integer $m \geq 2$ and vertices $u^{(c)}_1, u^{(c)}_2, \dots, u^{(c)}_{2m}$ appearing on $c_2c_1$, in that order from $c_2$ to $c_1$, there are vertices $v^{(c)}_1, v^{(c)}_2, \dots, v^{(c)}_{2m}$ appearing on $c_1x_3$, in that order from $c_1$ to $x_3$, and there are vertices $w^{(c)}_1, w^{(c)}_2, \dots,$ $w^{(c)}_{2m}$ appearing on $x_3c_2$, in that order from $x_3$ to $c_2$, such that segments $u^{(c)}_iu^{(c)}_{i+1},$ $v^{(c)}_iv^{(c)}_{i+1},$ $w^{(c)}_iw^{(c)}_{i+1}$ are minimal, and the \tbas~induced by $c_2u^{(c)}_1w^{(c)}_{2m},$ $c_1v^{(c)}_1u^{(c)}_{2m}$ and $x_3w^{(c)}_1v^{(c)}_{2m}$, are of type $\tbty{B}{0}$ or $\tbty{B}{1}$, with minimal initial segments $u^{(c)}_1w^{(c)}_{2m},$ $v^{(c)}_1u^{(c)}_{2m},$ $w^{(c)}_1v^{(c)}_{2m}$. The initial segments, outside the three small regions $c_2u^{(c)}_1w^{(c)}_{2m},$ $c_1v^{(c)}_1u^{(c)}_{2m},$ $x_3w^{(c)}_1v^{(c)}_{2m}$ are given by $u^{(c)}_{2i-1} w^{(c)}_{2m+2-2i},$ $v^{(c)}_{2i-1} u^{(c)}_{2m+2-2i},$ $w^{(c)}_{2i-1} v^{(c)}_{2m+2-2i}$, for $i = 1, 2, \dots, m$, and the blocking segments outside the three small regions are given by $w^{(c)}_{2i-1} u^{(c)}_{2l+2-2i},$ $u^{(c)}_{2i-1} v^{(c)}_{2m+2-2i},$ $v^{(c)}_{2i-1} w^{(c)}_{2m+2-2i}$, for $i = 1, 2, \dots, m$.\\

If we look at the \tba~induced by $a_3u^{(a)}_{2k}v^{(a)}_1$, which we know has type $\tbty{B}{0}$ or $\tbty{B}{1}$ with $u^{(a)}_{2k}v^{(a)}_1$ minimal, since $\Delta_b$ has type $\trt$, it follows that $a_3u^{(a)}_{2k}v^{(a)}_1$ actually is a minimal $\overline{\mathcal{S}}$-region, so $a_3u^{(a)}_{2k}$ is also minimal. Similarly, it follows that $c_1 v^{(c)}_1, b_1v^{(b)}_{2l}, a_2 w^{(a)}_1, c_2 w^{(c)}_{2m}$ and $b_3 u^{(b)}_1$ are minimal segments. Let $i_a, i_b, i_c$ be such that $a_3 = u^{(b)}_{i_a}, b_1 = v^{(c)}_{i_b}, c_2 = w^{(a)}_{i_c}$. Define 
\begin{itemize}
\item[] $u_i = u^{(a)}_i$, for $1\leq i \leq 2k$, and $u_{2k + i} = u^{(b)}_{i_a + i - 1}$, for $1\leq i \leq 2l + 1 - i_a$, 
\item[] $v_i = v^{(b)}_i$, for $1\leq i \leq 2l$, and $v_{2l + i} = v^{(c)}_{i_b + i - 1}$, for $1\leq i \leq 2m + 1 - i_b$,
\item[] $w_i = w^{(c)}_i$, for $1\leq i \leq 2m$, and $w_{2m + i} = w^{(a)}_{i_c + i - 1}$, for $1\leq i \leq 2k + 1 - i_c$.
\end{itemize}
Let $p_u = 2k + 2l + 1 - i_a, p_v = 2l + 2m + 1 - i_b, p_w = 2m + 2k + 1 - i_c$ (which are the lengths of these sequences). We now show that $p_u = p_v = p_w$. Let $p_{u,1}$ be the number of vertices $u_i$, whose initial segment $\not= x_1x_2$ crosses $x_2x_3$, let $p_{v,1}$ be the number of vertices $v_i$, whose initial segment $\not= x_2x_3$ crosses $x_3x_1$, let $p_{w,1}$ be the number of vertices $w_i$, whose initial segment $\not= x_3x_1$ crosses $x_1x_2$, and let $p_{u,2} = p_u - p_{u,1}, p_{v,2} = p_v - p_{v,1}, p_{w,2} = p_w - p_{w,1}$.\\
\indent Firstly, we show that $p_{u,1} = p_{u,2}$ (so by symmetry $p_{v,1}= p_{v,2}$ and $p_{w,1}=p_{w,2}$). Note immediately that the initial segment $u_1w_{p_w}$ crosses $x_1x_2$ and $x_1x_3$, and that the initial segment $u_{p_u} v_1$ crosses $x_1x_2$ and $x_2x_3$. Also, for any $i < p_u$, observe that, by the definition of type $\trt$, the initial segments $\not=x_1x_2$ through $u_i$ and $u_{i+1}$ cannot both intersect $x_1x_3$, nor can both intersect $x_2x_3$. To spell out details, looking at $\Delta_a$, if $i < 2k$, then one of $u_i, u_{i+1}$ has the other initial segment with a vertex on $x_1a_2$, and the second has the other initial segment with a vertex on $a_3a_2$. Write temporarily $q$ for this second vertex among $u_i, u_{i+1}$ and $r\in a_3a_2$ for the vertex such that $qr$ is initial. But $l(qr)$ crosses $a_3a_2$, and we must have $l(qr)$ cross $x_2x_3$, as otherwise we obtain a contradiction by Proposition~\ref{ForbStruct1}. We argue similarly for $i \geq 2k$, by considering $\Delta_b$. The claim follows.\\
\indent Secondly, we show $p_{u,1} = p_{v,2}$ (and by symmetry $p_{v,1} = p_{w,2}$ and $p_{w,1} = p_{u,2}$). But, if a segment crosses $x_1x_2$ at some $u_i$, and crosses $x_2x_3$ at a point $q$, then by minimality of $u_{p_u} v_1$ and $v_{p_v} w_1$, $q = v_j$ for some $j$. However, this is injective map $i \mapsto j$, so $p_{u,1} \geq p_{v,2}$, and by symmetry $p_{u,1} = p_{v,2}$. From these observations, it follows that $p_u = p_v = p_w$, and we may write $p$ for this common value.\\ 

Next, we show that initial and blocking segments at $u_i, v_i, w_i$ satisfy the conditions of the type $\trt$. As we have seen already, initial segments through $u_i$ are $u_1 w_{i_1}, u_3 w_{i_3}, \dots, u_{p-1} w_{i_{p-1}}$ and $u_2 v_{i_2}, u_4v_{i_4}, \dots, u_p v_{i_p}$, where $i_1, i_3, \dots,$ $i_{p-1} \in [p]$ are distinct and even, and $i_2, i_4, \dots, i_p \in[p]$ are distinct and odd. However, if $i_{j} > i_{j'}$ holds for some $j > j'$ of the same parity, then we obtain a contradiction by Proposition~\ref{ForbStruct1}. Hence, $i_{j} = p + 1 - j$, for all $j$, as desired, and a similar argument shows that all initial segments have desired structure. For blocking segments, observe that all initial and blocking segments between $x_1x_2$ and $x_2x_3$ through some $u_i$ are disjoint, so the blocking segments have the desired structure. The intersections structure follows from the structure of $\Delta_a, \Delta_b, \Delta_c$.\\  
\indent Finally, we know from before that the \tbas~induced by $x_1u_1w_{p}, x_2v_1u_{p}, x_3 w_1 v_{p}$ are of type $\tbty{B}{0}$ or $\tbty{B}{1}$, with minimal segments $u_1w_{p}, v_1u_{p}, w_1 v_{p}$. This completes the proof.\qed\end{proof}

\begin{corollary}\label{triangleS3}Let $n \geq 1$ be an integer and suppose that $(n-1)$-\ct~holds. Let $\Delta = (T, \mathcal{S}, \mathcal{B})$ be a \tba~of size $n$, and let $x_1, x_2$ and $x_3$ be the vertices of the triangle $T$. Suppose that there are vertices $a_3, b_3 \in x_1x_2$, appearing in order $x_1, b_3, a_3, x_2$, $c_2, a_2 \in x_1x_3$, appearing in order $x_1, c_2, a_2, x_3$, and $b_1, c_1 \in x_3x_2$, appearing in order $x_3, b_1, c_1, x_2$, such that $a_2a_3, b_1b_3, c_1c_2 \in \mathcal{S}$. Then $\Delta$ has type $\trt$.\end{corollary}

\begin{proof}Let $d_1 = b_1b_3\cap c_1c_2, d_2 = a_2a_3 \cap c_1c_2, d_3 = a_2a_3\cap b_1b_3$. By previous proposition, we may suppose that $d_1$ is not in the interior of $x_1a_2a_3$, and moreover that there are no triples of segments where each pair intersects and form a small triangle that satisfies the conditions of Proposition~\ref{triangleCoversProp}.\\
\indent By Lemma~\ref{triangularBlocksLemma}, we have that $x_1, x_2, x_3$ are externally blocked, and in regions $a_3b_3d_3, a_2c_2d_2, b_1c_1d_1$, the vertices $a_3, b_3, d_3, a_2, c_2, d_2, b_1, c_1, d_1$ are externally blocked as well. Suppose for a moment that $a_3d_3$ is not a minimal segment. Let $q \in a_3d_3$ be another vertex. Let $pr$ be another initial segment through $q$, with $p,r \in \partial T$, with $p \in a_3x_1 \cup x_1a_2$ and $r \in a_3x_2 \cup x_2x_3 \cup x_3a_2$. If $pr$ crosses $b_3b_1$, then either Proposition~\ref{ForbStruct1} gives a contradiction, for $pr$ and $b_3b_1$ or $pr$ and $a_3a_2$, or the segments $pr, b_3b_1, a_3a_2$ form a triple we forbade at the beginning of the proof. Hence, $pr \cap b_3b_1 = \emptyset$, so $p \in a_3b_3, r \in x_2b_1$. However, if we apply the $(n-1)$-\ct~to the \tba~induced by $x_2b_3b_1$, it must have type $\trt$, and it follows that there is an initial segment $s$ which crosses both $d_3a_3$ and $d_3b_3$. By Proposition~\ref{ForbStruct1}, it follows that $s$ crosses $x_1x_3$ and $x_2x_3$, thus $s, b_3b_1, a_3a_2$ forms a structure that we forbade at the beginning of the proof. This is a contradiction, and it follows similarly that $a_3d_3, b_3d_3, d_2a_2, d_2c_2, d_1b_1, d_1c_1$ are minimal, and further, $a_3b_3, a_2c_2, b_1c_1$ are minimal as well.\\
\indent Therefore, we have actually shown that any configuration of segments like $a_2a_3, b_1b_3, c_1c_2$ implies the minimality of the segments $a_3d_3, b_3d_3$, etc. Using this observation and Proposition~\ref{ForbStruct1}, it follows also that $d_1d_2, d_2d_3$ and $d_3d_1$ are minimal. Applying the $(n-1)$-\ct~to the \tbas~induced by $x_1a_3a_2, x_2b_3b_1$ and $x_3c_2c_1$, all three have the type $\tbty{B}{3}$, and there are vertices $a'_2 \in x_1c_2, a'_3 \in b_3x_1, b'_3 \in a_3x_2, b'_1 \in x_2c_1, c'_1 \in b_1x_3, c'_2 \in a_2x_3$, such that $c_2a'_2, a'_2a'_3, a'_3b_3, a_3b'_3, b'_3b'_1, b'_1c_1, b_1c'_1, c'_1c'_2, c'_2a_2$ are minimal initial segments, and $\beta(d_1) = b'_3c'_2, \beta(d_2) = c'_1a'_3, \beta(d_3) = a'_2b'_1$. From the types of same \tbas, it follows that \tbas~induced by $x_1a'_2a'_3, x_2b'_3b'_1, x_3c'_1c'_2$ are of type $\tbty{B}{0}$ or $\tbty{B}{1}$, with $a'_2a'_3, b'_1b'_3, c'_1c'_2$ minimal. Thus, $\Delta$ has the type $\trt$.\qed\end{proof}

\begin{proposition}\label{finalS3prop}Let $n \geq 1$ be an integer and suppose that $(n-1)$-\ct~holds. Let $\Delta = (T, \mathcal{S}, \mathcal{B})$ be a \tba~of size $n$, and let $x_1, x_2$ and $x_3$ be the vertices of the triangle $T$. Suppose that there are vertices $a_1 \in x_1x_2, a_2,b_1 \in x_1x_3, b_2 \in x_2x_3$ such that $a_1a_2$ and $b_1b_2$ are initial segments, intersecting at a point $c$. Then $\Delta$ has type $\trt$.\end{proposition}

\begin{proof} By Lemma~\ref{triangularBlocksLemma}, the vertices $x_1,x_2,x_3$ are externally blocked.  Let $c_1, c_2, \dots,$ $c_r$ be the vertices that lie on $a_1a_2$, in that order from $a_1$ to $a_2$. Thus, $r \geq 1$. Let the initial segment $\not=a_1a_2$ through $c_i$ be $p_iq_i$. Note that at least one of $p_i,q_i$ must be on $x_2x_3$, otherwise we obtain a contradiction using Proposition~\ref{ForbStruct1}, without loss of generality, $q_i \in x_2x_3$. Also, if $p_i \in x_1x_3, p_{i+1} \in x_1x_2$, then by Corollary~\ref{triangleS3}, $\Delta$ has type $\trt$. Thus, assume that there is $i_0$ such that $p_i \in x_1x_2$ for $i \leq i_0$ and $p_i \in x_1x_3$ for $i > i_0$. Moreover, by Proposition~\ref{ForbStruct1}, on $a_1x_1$, the vertices $a_1, p_1, p_2, \dots, p_{i_0}, x_1$ appear in this order, and on $x_1a_2$, the vertices $x_1 p_{i_0+1} p_{i_0 + 2} \dots p_r a_2$ appear in this order. By Lemma~\ref{triangularBlocksLemma}, we also have $a_1, p_i$ and $c_i$ externally blocked in the region $a_1p_ic_i$, for $i \leq i_0$, and $a_2, p_i$ and $c_i$ externally blocked in the region $a_2p_ic_i$ for $i > i_0$. However, applying the $(n-1)$-\ct~to the \tba~$\Delta_a$ induced by $a_1a_2x_1$, the only type that can be satisfied by $\Delta_a$ is $\tbty{B}{3}$.\\
\indent From the definition of type $\tbty{B}{3}$, it follows that $p_1c_1, c_1c_2, c_2p_2$ are minimal, that there are vertices $r_1 \in x_1p_1, r_2 \in x_1p_2$ such that $r_1r_2$ is a minimal initial segment, and also that $r_1p_1, r_2p_2$ are minimal. Furthermore, $r_2c_1, r_1c_2, p_1p_2$ are all blocking segments.\\
\indent Observe that, by the assumption \textbf{(A)}, if $l(c_1r_2)$ crosses $x_1x_2$, then $l(p_2c_2)$ must pass through the same point, however, $l(p_2c_2)$ cuts $x_2x_3$, which is a contradiction. Hence, $\beta(c_1)$ crosses the segment $x_2q_1$, and similarly, $\beta(c_2)$ cuts $q_2x_3$. Therefore, $(n-1)$-\ct~applied to the \tbas~$\Delta_1, \Delta_2$ induced respectively by regions $p_1x_2q_1$ and $p_2x_3q_2$, both have type $\trt$ or $\tbty{B}{3}$. In particular, looking at $\Delta_1$, there are vertices $u \in x_2q_1,v \in c_1q_1$, such that $uv$ is an initial segment. As $c_1c_2$ and $c_2a_2$ are minimal, $l(uv)$ is disjoint from $c_1a_2$, so $l(uv)$ must cross $a_2x_3$, at some point $w$. But, $uw$ and $p_2q_2$ must cross as well, and this is a final contradiction granted by Proposition~\ref{ForbStruct1}.\qed\end{proof} 


\begin{lemma}\label{S12basicintersectionLemma}Let $n \geq 1$ be an integer and suppose that $(n-1)$-\ct~holds. Let $\Delta = (T, \mathcal{S}, \mathcal{B})$ be a \tba~of size $n$, and let $x_1, x_2$ and $x_3$ be the vertices of the triangle $T$. Suppose that there are vertices $a_1, b_1 \in x_1x_2, a_2,b_2 \in x_2x_3$ such that $x_1, b_1, a_1, x_2$ and $x_3, a_2, b_2, x_2$ appear in these orders, and $a_1a_2, b_1b_2$ are initial segments, with an intersection point $c$. Then, in the region $a_1x_2b_2c$, the vertices $x_2$ and $c$ are internally blocked.\end{lemma}

\begin{proof} By Proposition~\ref{ForbStruct1}, we have $x_2$ internally blocked. Thus, $\Delta$ cannot have type $\trt$, so by Proposition~\ref{finalS3prop} it follows that any two initial segments that intersect have to have their endpoints on the same edges of $T$.\\
\indent Suppose contrary, vertex $c$ is externally blocked in $x_2a_1cb_2$. Thus, exactly one of $a_1, b_2$ is internally blocked in this region, by symmetry, we may assume that $a_1$ is internally blocked. So $b_1$ is internally blocked in $a_1b_1c$. Let $q = \beta(b_1) \cap a_1c$. Applying the $(n-1)$-\ct~to the \tba~induced by $x_2b_1b_2$, implies that it has type $\tbty{I}{1}$ or $\tbty{I}{2}$. Thus, there are vertices $p \in a_1b_1,r \in cb_2$, such that $p,q,r$ are collinear, and $pq, qr$ are minimal initial segments. Let $s$ be the intersection $l(pr) \cap b_2a_2$. However, $cb_2, qs$ intersect, while $a_2$ is externally blocked in $a_1a_2x_2$, so application of Proposition~\ref{ForbStruct1} results in contradiction.\qed\end{proof}

\begin{lemma}Let $\Delta$ and vertices $x_1,x_2,x_3,a_1,a_2,b_1,b_2,c$ satisfy the assumptions of Lemma~\ref{S12basicintersectionLemma}. Then, in the regions $a_1b_1c$ and $a_2b_2c$, the vertices $a_1,a_2,b_1,b_2,c$ are externally blocked.\end{lemma}

\begin{proof} As in the proof of Lemma~\ref{S12basicintersectionLemma}, intersecting initial segments must have endpoints on the same edges of $T$. Also, by that lemma, $x_2$ is internally blocked and $c$ is internally blocked in the region $a_1cb_2x_2$.\\
\indent Consider the vertex $b_1$. If we prove that $b_1$ is externally blocked in the $a_1b_1c$, then it follows that so is $a_1$, and looking at regions $a_1a_2x_2$ and $b_1b_2x_2$, the conclusion follows. Therefore, assume contrary, that $b_1$ is internally blocked in $a_1b_1c$.\\
\indent Set $q = \beta(b_1) \cap a_1c$. By $(n-1)$-\ct~applied to the \tba~$\Delta_b$ induced by $b_1b_2x_2$, the type of $\Delta_b$ is either $\tbty{I}{1}$ or $\tbty{I}{2}$, but in either case we have vertices $p \in a_1x_2, r \in b_1c$, such that $p,q,r$ are collinear, and $pq, qr$ are minimal initial segments. Looking at $l(pr)$ and $b_1b_2$, as these intersect, the line $l(pr)$ must cross $a_2x_3$, with intersection point $s$, say. Recalling that the type of $\Delta_b$ is either $\tbty{I}{1}$ or $\tbty{I}{2}$, we have that $p$ is externally blocked in $pqa_1$, hence in $x_2ps$, the vertices $x_2, p$ are internally blocked, and therefore $s$ is externally blocked. But $qa_2$ and $rb_2$ cross at $c$, which is a contradiction by Proposition~\ref{ForbStruct1} applied to $x_2ps$.\qed\end{proof}

\begin{lemma} Let $\Delta$ and vertices $x_1,x_2,x_3,a_1,a_2,b_1,b_2,c$ satisfy the assumptions of Lemma~\ref{S12basicintersectionLemma}. Then, $a_1c$ and $b_2c$ are minimal and $x_2c$ is blocking.\end{lemma}

\begin{proof} By previous Lemma, we have that $a_1,b_1, a_2, b_2, c$ are externally blocked in regions $a_1b_1c$ and $a_2b_2c$. Apply the $(n-1)$-\ct~to the \tba~$\Delta_b$ induced by $b_1b_2x_2$, thus $\Delta_b$ has type $\tbty{B}{1}$, $\tbty{I}{1}$ or $\tbty{I}{2}$. But, in either case, $a_1c$ is minimal. Similarly, $b_2c$ is minimal, and $\beta(x_2)$ must pass through $c$.\qed\end{proof}

\begin{lemma}\label{S12finalintersectionLemma} Let $\Delta$ and vertices $x_1,x_2,x_3,a_1,a_2,b_1,b_2,c$ satisfy the assumptions of Lemma~\ref{S12basicintersectionLemma}. Then, $a_2c$ and $b_1c$ are minimal. \end{lemma}

\begin{proof} Suppose contrary, $b_1c$ is not minimal. Thus, there is a vertex $q \in b_1c$. Let $pr \not= b_1b_2$ be the initial segment through $q$, with $p, r \in \partial T$. As $a_1c$ is minimal, without loss of generality, $p \in a_1b_1$. As in the proof of Lemma~\ref{S12basicintersectionLemma}, since $pr$ and $b_1b_2$ intersect, $r \in x_2x_3$. However, we may apply the previous lemma to $b_1b_2$ and $pr$, to obtain that $b_2q$ is minimal, which is a contradiction as $c \in b_2q$.\qed\end{proof}

\begin{corollary}\label{segmentsCorollary} Let $n \geq 1$ be an integer and suppose that $(n-1)$-\ct~holds. Let $\Delta = (T, \mathcal{S}, \mathcal{B})$ be a \tba~of size $n$, and let $x_1, x_2$ and $x_3$ be the vertices of the triangle $T$. If $\Delta$ is not of type $\trt$, then for every $\{j_1, j_2, j_3\} = [3]$ there are vertices $p_1, p_2, \dots ,p_k \in x_{j_1} x_{j_2}$ and $q_1, q_2, \dots, q_k \in x_{j_1}x_{j_3}$ such that $x_{j_1} p_1, p_i p_{i+1}, x_{j_1}q_1,  q_i q_{i+1}$ are minimal, for all $i < k$, and one of the following alternatives holds.
\begin{enumerate}
\item Vertex $x_{j_1}$ is externally blocked. Each $p_iq_i$ is a minimal initial segment, and every initial segment with one vertex on $x_{j_1}x_{j_2}$ and the other vertex on $x_{j_1}x_{j_3}$ is among $p_iq_i$. Furthermore, $k$ is odd, and $p_{2i-1}q_{2i}, p_{2i}q_{2i-1}$ are blocking segments, for $i \leq \frac{k-1}{2}$, and $p_k, q_k$ are externally blocked in $x_{j_1} p_k q_k$.
\item Vertex $x_{j_1}$ is externally blocked. Each $p_iq_i$ is a minimal initial segment, and every initial segment with one vertex on $x_{j_1}x_{j_2}$ and the other vertex on $x_{j_1}x_{j_3}$ is among $p_iq_i$. Furthermore, $k$ is even, and $p_{2i-1}q_{2i}, p_{2i}q_{2i-1}$ are blocking segments, for $i \leq \frac{k}{2}$.
\item Vertex $x_{j_1}$ is internally blocked and $k$ is even. For all $i \leq k/2$, the segments $p_{2i-1}q_{2i}$ and $p_{2i}q_{2i-1}$ are initial segments and intersect at point $r_i$. Every initial segment with one vertex on $x_{j_1}x_{j_2}$ and the other vertex on $x_{j_1}x_{j_3}$ is among these. The initial segments $r_ip_{2i-1}, r_ip_{2i}, r_iq_{2i-1}r_iq_{2i}$ are minimal. The vertices $r_1, \dots, r_{k/2}$ all lie on $\beta(x_{j_1})$. Also, $p_1q_1$ is blocking, and $p_{2i}q_{2i+1}, p_{2i+1}q_{2i}$ are blocking for $i < k/2$, and $p_k$ is externally blocked in $x_{j_1}p_kq_{k-1}$, and $q_k$ is externally blocked in $x_{j_1}q_kp_{k-1}$.
\end{enumerate}
\end{corollary}

\begin{proof} Without loss of generality, $j_1 = 1, j_2 = 2, j_3 = 3$. We split into two cases, depending on whether some initial segments between $x_1x_2$ and $x_1x_3$ intersect or not. The possible outcomes are shown in Figure~\ref{cor19Fig}.\\[6pt]
\begin{figure}
\begin{center}
\includegraphics[width = 0.8\textwidth]{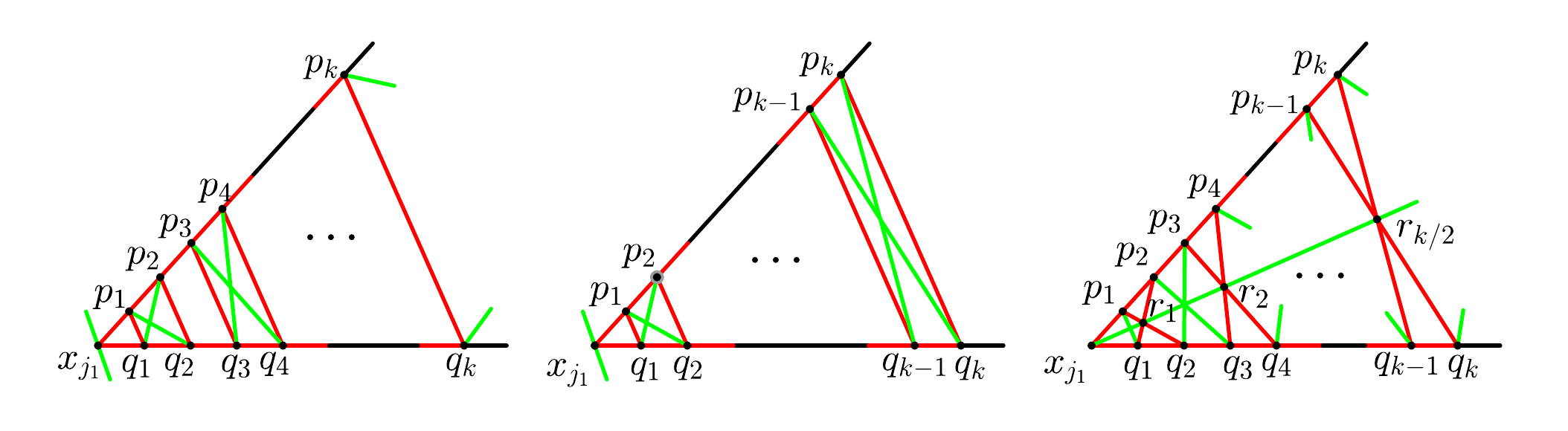}
\caption{Possibilities in Corollary~\ref{segmentsCorollary}}
\label{cor19Fig}
\end{center}
\end{figure}
\noindent\textbf{Case 1: there is an intersecting pair.} By Lemma~\ref{S12finalintersectionLemma}, $x_1$ is internally blocked. Let $pq$ be any initial segment with $p \in x_1x_2, q \in x_1x_3$. Then $\beta(x_1)$ crosses $pq$, at a point $r$, say, and let $p'q'$ be another initial segment through $r$, with $p',q' \in \partial T$. However, $\Delta$ is not of type $\trt$, and by Proposition~\ref{finalS3prop}, without loss of generality, $p' \in x_1x_2, q' \in x_1x_3$. Applying Lemma~\ref{S12finalintersectionLemma} to $pq$ and $p'q'$, it follows that $pr, p'r, qr, q'r, pp', qq'$ are minimal. Observe further that if an initial segment $s$ has an endpoint on $x_1p$, unless the second endpoint is on $x_1q$, $s$ crosses $pq$, and thus $s = p'q'$. Combining these observations, we conclude that there are points $p_1, p_2, \dots, p_k \in x_1x_2, q_1, q_2, \dots, q_k \in x_1x_3$, such that $x_1p_1, p_1p_2, \dots, p_{k-1}p_k, x_1q_1, q_1q_2, \dots, q_{k-1}q_k$ are minimal, $k$ is even, $p_{2i-1}q_{2i}, p_{2i}q_{2i-1}$ are initial segments, and every initial segment with a vertex on $x_1x_2$ and another vertex on $x_1x_3$ is one of $p_{2i-1}q_{2i}, p_{2i}q_{2i-1}$. Furthermore, $p_{2i-1}q_{2i}, p_{2i}q_{2i-1}$ intersect at a point $r_i$, and $r_ip_{2i-1}, r_ip_{2i}, r_iq_{2i-1}, r_iq_{2i}$ are minimal for all $i \leq k/2$.\\
\indent From the information about minimal segments, we are forced to have $r_1, r_2, \dots,$ $r_{k/2} \in \beta(x_1)$, $p_1q_1$ blocking, in the minimal $\overline{\mathcal{S}}$-region $p_{2i}r_iq_{2i} q_{2i+1}$ $r_{i+1} p_{2i+1}$, all six vertices are internally blocked, for every $1 \leq i < k/2$, and finally $p_k$ is externally blocked in $x_1 p_k q_{k-1}$ and $q_k$ is externally blocked in $x_1 q_k p_{k-1}$. It remains to prove that for $i < k/2$, $p_{2i}q_{2i+1}$ and $p_{2i+1}q_{2i}$ are blocking.\\
Suppose contrary, there is some $i < k/2$, such that, without loss of generality, $p_{2i}q_{2i+1}$ is not blocking. Looking at minimal $\overline{\mathcal{S}}$-region $p_{2i}r_iq_{2i}q_{2i+1} r_{i+1}p_{2i+1}$, it follows that $p_{2i} q_{2i}$ is blocking. However, looking at $\beta(p_{2i})$ and $l(r_i q_{2i})$, which meet at $q_{2i}$, by assumption \textbf{(A)}, it follows that $l(p_{2i+1} r_{i+1})$ also contains $q_{2i}$, which is a contradiction.\\[6pt]
\noindent\textbf{Case 2: there are no intersecting pairs.} Observe that if $p \in x_1x_2, q \in x_1x_3$ and $pq$ is an initial segment, then, if any other segment $s$ crosses $pq$, then, by the assumption of this case, $s$ must have at least one vertex on $x_2x_3$, which is impossible by Proposition~\ref{finalS3prop}. Thus, $pq$ is minimal, and as in the previous case, if a segment $s$ has an endpoint in $x_1p$, its other endpoint is bound to be in $x_1q$. From this, we conclude that there are points $p_1, p_2, \dots,$ $p_k \in x_1x_2,$ $q_1, q_2, \dots,$ $q_k \in x_1x_3$, such that $x_1p_1, p_1p_2, \dots,$ $p_{k-1}p_k, x_1q_1,$ $q_1q_2, \dots,$ $q_{k-1}q_k$ are minimal, and $p_iq_i$ is a minimal initial segment for every $i \leq k$, and if $s$ is an initial segment with endpoints on $x_1x_2$ and $x_1x_3$, then $s = p_iq_i$ for some $i$. As $x_1p_1, x_1q_1$ are minimal, $p_1, q_1$ are externally blocked in $x_1p_1q_1$. Thus, $p_1q_2$ and $p_2q_1$ are blocking segments. Hence, in the region $p_2q_2q_3p_3$, $p_2$ and $q_2$ externally blocked and so are $p_3, q_3$. Proceeding in this fashion, the conclusion of the corollary follows.\qed\end{proof}

\begin{proposition}\label{FinalInternalBlocks}Let $n \geq 1$ be an integer and suppose that $(n-1)$-\ct~holds. Let $\Delta = (T, \mathcal{S}, \mathcal{B})$ be a \tba~of size $n$, and let $x_1, x_2$ and $x_3$ be the vertices of the triangle $T$. Suppose that $x_1, x_2$ are internally blocked. Then, $\Delta$ has one the types $\tbty{B}{1},\tbty{I}{1}$ or $\tbty{I}{2}$.\end{proposition}

\begin{proof} We consider three cases, depending on the outcomes of Corollary~\ref{segmentsCorollary}. We say that $x_i$ \emph{has no segments} if there are no segments with one vertex on $x_ix_{i'}$ and the other on $x_ix_{i''}$, where $i \in \{1,2\}$ and $\{i, i', i''\} = \{1,2,3\}$. Otherwise, we say that $x_i$ \emph{has segments}.\\[6pt]
\noindent\textbf{Case 1: both $x_1, x_2$ have no segments.} Let $u = \beta(x_1) \cap x_2x_3, v = \beta(x_2) \cap x_1x_3$. By Corollary~\ref{segmentsCorollary}, any initial segment in $\mathcal{S}$ is of the form $pq$, where $p \in x_1x_3, q \in x_2x_3$ and all these are minimal (and hence disjoint). In particular, no initial segment can cross $x_1u, x_2v$, and also, $x_1v, x_2u$ are minimal segments. Thus, the other initial segment through $u$, must cross $x_3v$, and the other initial segment through $v$ must cross $x_3u$. However, all initial segments are disjoint, so actually $uv$ is an initial segment, and it is minimal. It follows from Corollary~\ref{segmentsCorollary} at vertex $x_3$ that the type of $\Delta$ is $\tbty{B}{1}$.\\[6pt]
\noindent\textbf{Case 2: $x_1$ has, but $x_2$ has no segments.} By Corollary~\ref{segmentsCorollary}, we have vertices $a_1, a_2, \dots,$ $a_k \in x_1x_3,$ $b_1, b_2, \dots,$ $b_k \in x_1x_2$ such that $k$ is even and $x_1a_1, a_1a_2, \dots,$ $a_{k-1}a_k, x_1b_1, b_1b_2, \dots,$ $b_{k-1}b_k, b_kx_2$ are minimal, $a_{2i-1}b_{2i}, a_{2i} b_{2i-1}$ are initial segments that intersect at a point $c_i$. Let $u = \beta(x_1) \cap x_2x_3$. By Corollary~\ref{segmentsCorollary} applied to vertex $x_3$, we see that every initial segment is either among $a_ib_j$, or has vertices on $x_1x_3$ and $x_2x_3$ and is minimal. Hence, $x_2u$ is minimal. Let $v \in x_1x_3$ be such that $uv$ is an initial segment, and thus minimal. Hence, $va_k$ is also minimal, as otherwise, an initial segment with a vertex on $va_k$ would have the second endpoint on $x_2x_3$, so it would have to be minimal, but would cross $\beta(x_1) = x_1u$ or $vu$, which is impossible. Finally, $a_kc_{k/2}b_kx_2uv$ is a minimal $\overline{\mathcal{S}}$-region. Using assumption \textbf{(A)} as before, we see that $x_2a_k, vb_k$ are blocking, and it follows that the type of $\Delta$ is $\tbty{I}{1}$.\\[6pt]
\noindent\textbf{Case 3: both $x_1, x_2$ have segments.} By Corollary~\ref{segmentsCorollary}, we have vertices $a_1, a_2, \dots,$ $a_{k} \in x_1x_3$, $b_1, b_2, \dots,$ $b_{k}, d_1, d_2, \dots,$ $d_l \in x_1x_2$, $e_1, e_2, \dots,$ $e_l \in x_2x_3$ such that $x_1a_1, a_1a_2, \dots,$ $a_{k-1}a_k$, $x_1b_1, b_1b_2, \dots,$ $b_{k-1}b_k$, $x_2d_1, d_1d_2, \dots,$ $d_{l-1}d_l$, $x_2e_1, e_1e_2, \dots,$ $e_{l-1}e_l$ are minimal, and $a_1b_2, a_2b_1, \dots, a_{k-1}b_k, a_kb_{k-1}$, $d_1e_2, d_2e_1, \dots,$ $d_{l-1}e_l, d_l e_{l-1}$ are initials segments, $k$, $l$ are even and these are all initial segments that have at least one vertex on $x_1x_2$. Furthermore, $a_{2i-1}b_{2i}, a_{2i}b_{2i-1}$ intersect at point $c_i$, $d_{2i-1}e_{2i}, d_{2i}e_{2i-1}$ intersect at a point $f_i$, such that $c_1, c_2, \dots,$ $c_{k/2} \in \beta(x_1),$ $f_1, f_2, \dots,$ $f_{l/2} \in \beta(x_2)$. We also have that $b_kd_l$ is minimal.\\
\indent Suppose for a moment that $a_kx_3$ is minimal. Then, as $e_lf_{l/2},$ $f_{l/2}d_l,$ $d_lb_k,$ $b_k c_{k/2},$ $c_{k/2}a_k$ are minimal, it follows that $e_l x_3$ is also minimal, and hence $a_kc_{k/2}b_kd_l f_{l/2}e_lx_3$ is a minimal $\mathcal{L}$-region. But inside this minimal region, $\beta(f_{l/2})$ can only pass through $a_k$, and $\beta(c_{k/2})$ can only pass through $e_l$. However, then we have $\beta(b_k) = \beta(d_l)$, which is a contradiction. Therefore, $a_kx_3$ is not minimal.\\   
\indent Let $v$ be the vertex in $a_kx_3$ such that $a_kv$ is minimal. As in the previous case, any initial segment with vertices on $x_1x_3$ and $x_2x_3$ is minimal. It follows that the initial segment through $v$, not equal to $x_1x_3$, is $vu$ with $u \in x_3e_l$ and it is minimal. Since $uv, va_k, a_kc_{k/2}, c_{k/2}b_k, b_kd_l, d_lf_{l/2}, f_{l/2}e_l$ are minimal, it follows that $ue_l$ is minimal as well. Therefore, $R = a_kc_{k/2}b_kd_lf_{l/2}e_luv$ is a minimal $\overline{\mathcal{S}}$-region. Using Corollary~\ref{segmentsCorollary}, to prove that $\Delta$ has type $\tbty{I}{2}$, it suffices to show that $a_k \in \beta(x_2), e_l \in \beta(x_1)$.\\
\indent Suppose contrary, that $\beta(x_1) \cap x_2x_3 \not= e_l$. By minimality of $R$, we must have $\beta(x_1) \cap x_2x_3 = u$. But, we would then have $\beta(u) \cap l(va_k) = x_1$. By the assumption \textbf{(A)}, it follows that $l(e_lf_{l/2})$ also passes through $x_1$, which is impossible. Thus $e_l \in \beta(x_1)$, and similarly we obtain $a_k \in \beta(x_2)$. Thus, $\Delta$ has type $\tbty{I}{2}$.\qed\end{proof}

Finally, it remains to classify the \tbas~without intersecting initial segments.

\begin{proposition} \label{DisjointSegmentsProp} Let $n \geq 1$ be an integer and suppose that $(n-1)$-\ct~holds. Let $\Delta = (T, \mathcal{S}, \mathcal{B})$ be a \tba~of size $n$, and let $x_1, x_2$ and $x_3$ be the vertices of the triangle $T$. Suppose that no two initial segments in $\mathcal{S}$ intersect. Then $\Delta$ has a basic type or $\trt$ (with $k=1$ in the definition of $\trt$). \end{proposition}

\begin{proof} We say that a vertex $x_i$ is empty, if for $\{i,j,j'\} = \{1,2,3\}$, there are no initial segments between $x_ix_j$ and $x_ix_{j'}$. We distinguish between four cases, depending on the number of empty vertices.\\[3pt]
\noindent\textbf{Case 0: All four vertices are empty.} Then $\mathcal{L} = \emptyset$, and $\Delta$ has type $\tbty{B}{0}$.\\[3pt]
\noindent\textbf{Case 1: Only $x_1$ is non-empty.} Applying Corollary~\ref{segmentsCorollary}, $\Delta$ has type $\tbty{B}{1}$.\\[3pt]
\noindent\textbf{Case 2: Vertices $x_1, x_2$ are non-empty.} By Corollary~\ref{segmentsCorollary}, there are vertices $a \in x_1x_3,b,c \in x_1x_2,d \in x_2x_3$ such that $ab, cd$ are initial segments and all initial segments are either in $x_1ab$ or in $x_2cd$, and are disjoint. Moreover, we obtain the desired structure of blocking lines in regions $x_1ab, x_2cd$. Moreover, $bc, dx_3, ax_3$ are minimal, so $R = x_3abcd$ is a minimal $\overline{\mathcal{S}}$-region, inside which $a$ is internally blocked iff $b$ is, and $c$ is internally blocked iff $d$ is. If none of these four vertices are internally blocked in $R$, $\Delta$ is of type $\tbty{B}{2}$. Assume for contradiction that some vertex among them is internally blocked in $R$. Without loss of generality, one of $a,b$ is internally blocked, so both of them must be internally blocked. However, $x_3$ is externally blocked in $R$, so $\beta(b)$ must pass through $d$. However, we have $l(cd) \cap \beta(b) = d$, so by the assumption \textbf{(A)} applied to $R$ and vertex $b$, we have that $l(ax_3)$ passes through $d$, which is a contradiction, as desired.\\[3pt]
\noindent\textbf{Case 3: All three vertices are non-empty.} Similarly to the previous case, there are vertices $a,f \in x_1x_3, b,c \in x_1x_2, d,e \in x_2x_3$ such that $ab, cd, ef$ are minimal initial segments, and all initial segments are in regions $x_1ab,$ $x_2cd,$ $x_3ef$, and are minimal. From this, $fa, bc, de$ are also minimal. Moreover, we know the structure of blocking lines in $x_1ab, x_2cd, x_3ef$, and it remains to determine the structure of blocking lines in the minimal $\overline{\mathcal{S}}$-region $R = abcdef$.\\
\indent We have that in $R$, $a$ is internally blocked iff $b$ is, $c$ is internally blocked iff $d$ is, and $e$ is internally blocked iff $f$. If all these are externally blocked, then $\Delta$ has $\tbty{B}{3}$, as desired. Now, assume that, without loss of generality, one of vertices $a,b$ is internally blocked in $R$. But then both $a$ and $b$ must be internally blocked. Suppose for a moment that there is a blocking segment in $R$, which is a small diagonal of hexagon $abcdef$. By symmetry, we may suppose it contains $a$, so it is $ac$ or $ae$. If $ac$ is blocking, however, $\beta(a) \cap l(bc) = c$, so by the assumption \textbf{(A)}, $l(ef)$ has to contain $c$, which is impossible. Similarly, if $ae$ is blocking, $\beta(a)\cap l(fe) = e$, so by the assumption \textbf{(A)}, $l(bc)$ has to contain $e$, which is also impossible.  Hence, the only possible blocking segments in $R$ are the main diagonals $ad, be, cf$. As $a,b$ are internally blocked, we have that $ad, be$ are blocking segments. But, as $d$ is internally blocked in $R$, so is $c$, so $cf$ is also blocking, showing that $\Delta$ has type $\tbty{B}{3}$.\qed\end{proof}

Combining all ingredients, we are ready to prove the \ct.

\begin{proof}[Proof of the \ct.] We prove the theorem by induction on the size $n$ of \tba. The base of induction is $n = 0$, when the \tba~has type $\tbty{B}{0}$.\\
\indent Now, assume that $n\geq 1$ and $(n-1)$-\ct~holds and let $\Delta$ be a \tba~of size $n$ with vertices $x_1,x_2,x_3$. If $\Delta$ has intersecting initial segments that satisfy the conditions of Proposition~\ref{finalS3prop}, we are done. Otherwise, if there are any intersecting initial segments in $\Delta$ at all, by Proposition~\ref{ForbStruct1}, we must have some of $x_1, x_2,x_3$ internally blocked. Then, we are done by Proposition~\ref{FinalInternalBlocks}. Finally, if there are no intersecting initial segments in $\Delta$, we may apply Proposition~\ref{DisjointSegmentsProp} to finish the proof.\qed\end{proof}

\section{Concluding remarks}

Our first remark is that it would also be very interesting to classify all triangle blocking arrangements, without the assumption \textbf{(A)}. However, this is probably much harder, as the following discussion suggests.\\

\noindent\textbf{A comment about the assumption \textbf{(A)}.} As we have seen before, the assumption \textbf{(A)} is necessary in the \ct. However, there could be hope that we are using this assumption only locally, and that the arrangement types are rigid enough so that after some point, large arrangements are forced to combine as in the \ct. However, Figure~\ref{TBAnotInDef} shows that we cannot localize the assumption \textbf{(A)}. Namely, a natural weaker assumption would be that for some fixed $K$, and for any minimal $\bar{\mathcal{S}}$-region $R$, for any consecutive vertices $v_1, v_2, v_3, v_4, v_5$ appearing in this order on $\partial R$, we have that, if $l(v_1v_2)$ and $\beta(v_3)$ intersect in $T$ as some point $p$, $\beta(v_3)$ meets the interior of $R$, and $v_2p$ or $v_3p$ has at most $K$ points, then $l(v_1v_2), \beta(v_3), l(v_4v_5)$ are concurrent. But this figure shows that we may have as many points between as we want; the only region where \textbf{(A)} fails is $abcdef$, namely $l(ab), \beta(c)$ meet at $x_1$, but $x_1 \notin l(de)$, and this region satisfies the weaker assumption.\\[6pt]
\noindent\textbf{Relationship with Green Tao theorem on ordinary lines.} We discuss very briefly the proof of the result about ordinary lines of Green and Tao~\cite{GreenTao}. It can be summarized as follows.
\begin{itemize}
\item[\textbf{Step 1.}] Move to the dual.
\item[\textbf{Step 2.}] Apply Melchior's inequality (which is a consequence of Euler's formula) to get some control over point-line incidences.
\item[\textbf{Step 3.}] Use the incidence information to find large pieces with `triangular structure'.
\item[\textbf{Step 4.}] Study `triangular structure' to show that it looks like a hexagonal lattice.
\item[\textbf{Step 5.}] Apply the dual version of Chasles' theorem to place the points on a cubic. 
\end{itemize}

\textbf{Step 4} corresponds to our \ct, and to emphasize the similarity, we phrase it as the following Classification Lemma. The conclusion is written slightly informally.

\begin{lemma}[Classification of triangular arrangements, Green and Tao~\cite{GreenTao}.] Let $T = x_1x_2x_3$ be a triangle in the plane, and let $\mathcal{S}$ be a collection of segments with endpoints on $\partial T$ with the property that whenever two segments in $\bar{\mathcal{S}} = \mathcal{S} \cup \{x_1x_2, x_2x_3, x_3x_1\}$ intersect, there is a unique third segment in $\bar{\mathcal{S}}$ that contains the intersection point, except possibly if the intersection is one of $x_1, x_2, x_3$,  in which case there might not be the third segment. Then, $\bar{\mathcal{S}}$ forms a hexagonal grid shown in Figure~\ref{hexagonalGridStructFig}. 
\end{lemma}
\begin{figure}
\begin{center}
\includegraphics[width = 0.3\textwidth]{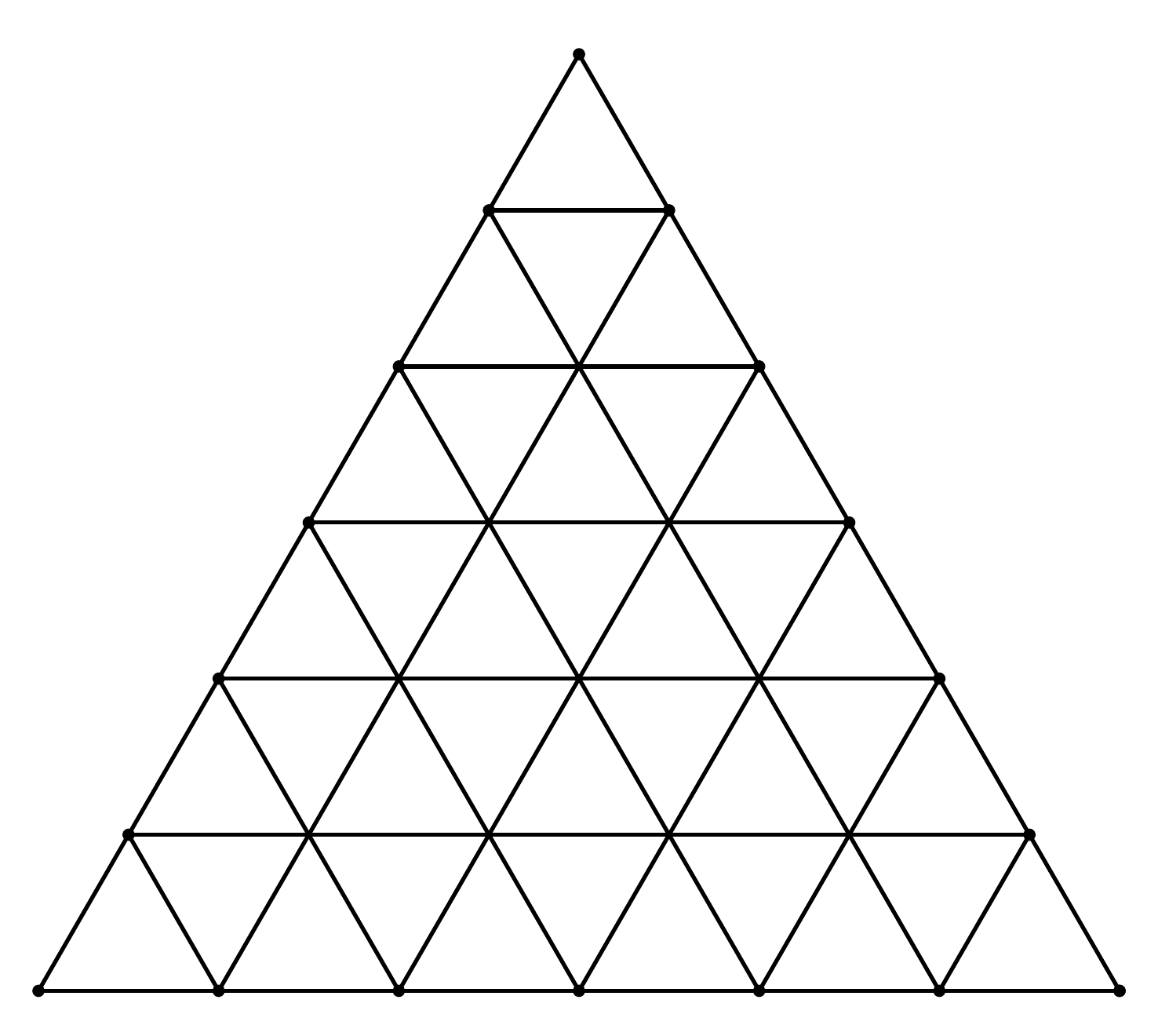}
\caption{Hexagonal grid}
\label{hexagonalGridStructFig}
\end{center}
\end{figure}
\begin{proof} We prove the claim by induction on $|\mathcal{S}|$. If $\mathcal{S}$ is empty, we are done. Assume now that we are given $\mathcal{S}$ and the claim holds for all smaller arrangements.\\
\indent Observe immediately that if $v$ is an intersection point on some edge of $T$, but not among the vertices $x_1, x_2, x_3$, then, we have $u, w \in \partial T$ such that $uv, wv \in \mathcal{S}$. Without loss of generality, $v \in x_1x_2, u \in x_1x_3$. If $w \in x_1x_3$ also, then, without loss of generality, $u$ is between $x_1$ and $w$, so applying the induction hypothesis to $vx_1w$ gives a contradiction, as hexagonal structure does not allow three segments at $v$. Hence, we must always have the two segments that meet on $\partial T$ between different pairs of edges of $T$.\\
\indent Similarly, we show that if two segment intersect, then they are between different pairs of edges of $T$. Suppose for the sake of contradiction that $a, c \in x_1x_2$ and $b,d \in x_1x_3$ are such that $ab, cd$ intersect at $e$. Without loss of generality, $c$ is between $x_1$ and $a$. Applying the induction hypothesis to $x_1ab$, we obtain a segment $ef$ with $f \in x_1b$. But, applying the induction hypothesis to $x_1cd$, we obtain a contradiction.\\
\indent Without loss of generality, we have a segment between $x_1x_2$ and $x_1x_3$. Pick an endpoint $v \in x_1x_2$ of such a segment with the property that $v$ is closest to $x_2$ among all such points. Let $u \in x_1x_3$ be such that $uv \in \mathcal{S}$. By observations before, there are no other points in $vx_2$ and all the segments between $x_1x_2$ and $x_1x_3$ are in $ux_1v$. In particular, $ux_3$ also has no points in its interior. We may apply the induction hypothesis to $x_1vu$, to obtain hexagonal structure there, with points $w_1, w_2, \dots, w_k$ appearing from $v$ to $u$. Consider segment $w_1b_1$ with $b_1 \in x_1u$. Then $l(w_1b_1)$ must cross $x_2x_3$, at some point $t_1$. But, then at $t_1$ we also have a segment with other endpoint on $x_1x_2$. However, by the choice of $w_1$, this may only be $v$. Next, consider $w_2$, and apply the same argument. We obtain a point $t_2 \in x_2x_3$ such that $t_2w_2$ is a part of a segment with other endpoint on $x_1x_3$, so similarly we obtain $t_2w_1$ is a subset of a segment in $\mathcal{S}$. Proceeding further in this fashion, we eventually obtain the hexagonal grid.\qed\end{proof}

It is therefore plausible that an extremal result could be proved with a similar general strategy, but given the significant differences in the difficulty of the relevant Classification Theorem, we expect that the new interesting difficulties will arise, in particular because not all types we defined come from duals of points on cubic curves. Nevertheless, we will investigate this further.\\[6pt]
\noindent \textbf{Classification Theorem for curves in the plane.} Going back to the proof of~\ct, we made a heavy use of topological properties of the real plane. However, we mainly focused on order of points on a line, and did not rely too much on the fact that the lines are straight (except that at intersection points the lines change sides with respect to one another). Instead of asking what happens over a different field, it could be possible that a similar, if not the same theorem holds for curves instead lines. Here we need some conditions on the curves, e.g. that we have some family of curves $\mathcal{C}$ with the property that through any two distinct points, there is a unique line in $\mathcal{C}$ containing them. Then, we could consider configurations where segments are intersections of curves in $\mathcal{C}$ with $T$. Or, we might not need to go that far and maybe we could consider curves with endpoints on $\partial T$ which are not self-intersecting and any two intersect in at most one point. This is something we shall also study further.\\
\indent Returning to the possibility of using a different field, this is of course another interesting question. However, over $\mathbb{C}$ we have, for example, the Hesse configuration (which can be realized as inflection points of a cubic curve), which gives 9 points, without ordinary lines. In this setting the interesting phenomenon is actually of a different nature.

\begin{theorem}[Kelly,~\cite{Kelly}] Any finite set of points in a complex space without ordinary lines is coplanar.\end{theorem}

We also expect that a classification theorem over finite fields would be very different from the one proved here.

\subsection{Acknowledgements}
\hspace{12pt}I would like to thank Trinity College and the Department of Pure Mathematics and Mathematical Statistics of Cambridge University for their generous support while this work was carried out and Imre Leader for helpful discussions concerning this paper. I also acknowledge the support of the Ministry of Education, Science and Technological Development of the Republic of Serbia, Grant III044006.

\end{document}